\documentclass[11pt]{article}

\usepackage{amsmath,amsthm,amssymb,amsbsy}

\usepackage{graphicx,caption,subfig}

\usepackage{hyperref}

\usepackage{fullpage}

\usepackage{float}

\usepackage[normalem]{ulem}

\usepackage[colorinlistoftodos]{todonotes}

\usepackage{prettyref}
\newrefformat{def}{Definition \ref{#1}}
\newrefformat{thm}{Theorem \ref{#1}}
\newrefformat{prop}{Proposition \ref{#1}}
\newrefformat{lem}{Lemma \ref{#1}}
\newrefformat{cor}{Corollary \ref{#1}}
\newrefformat{sec}{Section \ref{#1}}
\newrefformat{sub}{Subsection \ref{#1}}
\newrefformat{rem}{Remark \ref{#1}}
\newrefformat{ex}{Example \ref{#1}}
\newrefformat{app}{Appendix \ref{#1}}
\newrefformat{subfig}{Figure \ref{#1}}
\newrefformat{fig}{Figure \ref{#1}}

\numberwithin{equation}{section} \numberwithin{figure}{section}

\theoremstyle{plain} \newtheorem{thm}{\protect\theoremname}[section]
\theoremstyle{definition} \newtheorem{defn}[thm]{\protect\definitionname}
\theoremstyle{plain} \newtheorem{prop}[thm]{\protect\propositionname}
\theoremstyle{plain} \newtheorem{cor}[thm]{\protect\corollaryname}
\theoremstyle{plain} \newtheorem{lem}[thm]{\protect\lemmaname}
\theoremstyle{plain} 
\theoremstyle{remark} \newtheorem{rem}[thm]{\protect\remarkname}

\providecommand{\corollaryname}{Corollary}
\providecommand{\definitionname}{Definition}
\providecommand{\lemmaname}{Lemma}
\providecommand{\propositionname}{Proposition}
\providecommand{\theoremname}{Theorem}
\providecommand{\examplename}{Example}
\providecommand{\remarkname}{Remark}

\global\long\def\xvec{\boldsymbol{x}}
\global\long\def\cvec{\boldsymbol{c}}
\global\long\def\VV{\mathbf{V}}
\global\long\def\O{\Omega}
\newcommand{\w}{\omega}
\newcommand{\CC}{\mathbb{C}}
\newcommand{\NN}{\mathbb{N}}
\newcommand{\RR}{\mathbb{R}}
\newcommand{\GG}{\mathbf{G}}
\newcommand{\DD}{\mathcal{D}}
\newcommand{\srf}{\ensuremath{\mathrm{SRF}}}
\newcommand{\xiv}{\boldsymbol{\xi}}

\newcommand{\tab}{\hspace*{2em}}

\newcommand{\grd}[1][\Delta]{\ensuremath{{\cal T}_{#1}}}
\newcommand{\grdM}{{\cal R}}
\newcommand{\minmax}{{\cal E}}

\DeclareMathOperator{\sinc}{sinc}
\DeclareMathOperator{\diag}{diag}
\DeclareMathOperator{\Arg}{Arg}
\DeclareMathOperator{\supp}{supp}

\usepackage{constants}

\usepackage{authblk}

\begin{document}

\title{Conditioning of partial nonuniform Fourier matrices with
  clustered nodes\thanks{The research of DB and LD is supported in
    part by AFOSR grant FA9550-17-1-0316, NSF grant DMS-1255203, and a
    grant from the MIT-Skolkovo initiative. The research of GG and YY
    is supported in part by the Minerva Foundation.} }

\author[1]{Dmitry Batenkov} \author[1]{Laurent Demanet}
\author[2]{Gil Goldman} \author[2]{Yosef Yomdin}

\affil[1]{Department of Mathematics, Massachusetts Institute of
  Technology, Cambridge, MA 02139, USA} \affil[2]{Department of
  Mathematics, Weizmann Institute of Science, Rehovot 76100, Israel}

\affil[ ]{\textit{\{batenkov,ldemanet\}@mit.edu,
  \{gil.goldman,yosef.yomdin\}@weizmann.ac.il}}


\maketitle

\begin{abstract}
  We prove sharp lower bounds for the smallest singular value of a
  partial Fourier matrix with arbitrary ``off the grid'' nodes
  (equivalently, a rectangular Vandermonde matrix with the nodes on
  the unit circle), in the case when some of the nodes are separated
  by less than the inverse bandwidth. The bound is polynomial in the
  reciprocal of the so-called ``super-resolution factor'', while the
  exponent is controlled by the maximal number of nodes which are
  clustered together. As a corollary, we obtain sharp minimax bounds
  for the problem of sparse super-resolution on a grid under the
  partial clustering assumptions.
\end{abstract}

{\bf Keywords:} {\it Vandermonde matrix with nodes on the unit circle,
  prolate matrix, partial Fourier matrix, super-resolution, singular
  values, decimation}

  {\bf AMS 2010 Subject Classification:} {{\it Primary:} 15A18; 
    {\it Secondary:} 42A82, 65F22, 94A12 }
\section{Introduction}
\label{sec:intro}

Vandermonde matrices and their spectral properties are of considerable
interest in several fields, such as polynomial interpolation,
approximation theory, numerical analysis, applied harmonic analysis,
line spectrum estimation, exponential data fitting and others
(e.g. \cite{aubel_vandermonde_2017,batenkov_complete_2015,bazan_conditioning_2000,pereyra2010,peter_nonlinear_2011,reynolds_generalized_2013}
and references therein). Motivated by questions related to the
so-called problem of super-resolution (more on this in
\prettyref{sub:discussion} below), in this paper we study the
conditioning of rectangular Vandermonde matrices $\VV$ with
irregularly spaced nodes on the unit circle, where the number of nodes
$s$ is considered to be relatively small and fixed, while the
polynomial degree $N\geq s$ can be large. This question has received
much attention in the literature, see
e.g. \cite{aubel_vandermonde_2017,bazan_conditioning_2000,liao_music_2016,moitra_super-resolution_2015,ferreira_superresolution_1999,li_stable_2017,batenkov_stability_2016,demanet_recoverability_2014}. Normalizing
the matrix by ${1\over\sqrt{N}}$, the magnitude of the largest singular value is
$O(\sqrt{s})$, and so studying the scaling of the condition number is
equivalent to estimating the smallest singular value. As long as the
nodes are separated by at least ${1\over N}$, the matrix $\VV$ is
known to be well-conditioned. However, as the nodes collide, the
columns of $\VV$ become increasingly correlated and therefore the
smallest singular value becomes very small, while the condition number
blows up.

In this paper we show (see \prettyref{sec:main-result-intro}) that if
the nodes are separated by $\Delta \ll {1\over N}$, then under certain
technical conditions the smallest singular value of $\VV$ scales with
the asymptotically tight rate $\asymp \left(N\Delta\right)^{\ell-1}$,
where $\ell \leq s$ is the maximal number of nodes which form a small
``cluster'' (i.e. a group of at most $\ell$ nodes which are separated below
$\sim {1\over N}$, see \prettyref{def:partial-cluster}). This improves
upon previous known results
\cite{demanet_recoverability_2014,li_stable_2017} which established
this scaling for the extreme case $\ell=s$, and a recent preprint
\cite{kunis2018} which deals with the special case $\ell=2$. During
the review of the present paper, the authors of \cite{li_stable_2017}
improved their analysis to the general case $\ell \leq s$, and we
compare their results to ours in \prettyref{rem:compare-with-LL}
below.

The above bounds follow from the solution of the ``continuous''
version of the problem, where the row index becomes a continuous
``frequency'' variable $\w \in \left[-\O,\O\right]$, so that the
bandwidth $\O$ effectively plays the role of $N$. In the continuous
setting, we establish tight bounds for the smallest eigenvalue of the
corresponding Gramian matrix $\GG$ with irregularly spaced nodes,
which generalizes well-known results due to Slepian
\cite{slepian_prolate_1978} for the prolate matrix (which, in turn,
plays a prominent role in the seminal study of the spectral
concentration problem \cite{slepian_comments_1983}). In fact this
continuous version is what originally appeared in the studies of the
super-resolution of sparse atomic measures in
\cite{donoho_superresolution_1992} and later
\cite{demanet_recoverability_2014}, and we use our results to derive
minimax bounds for this problem in \prettyref{sub:discussion}.

The paper is organized as follows. In \prettyref{sec:known-bounds} we
provide the definitions and review known bounds for singular values of
rectangular Vandermonde matrices. In \prettyref{sec:main-results} we
state the definition for clustered configurations, and formulate the
main results regarding the smallest eigenvalue of the Gramian matrix
$\GG$, smallest singular value of the corresponding Vandermonde matrix
$\VV$ and the novel minimax bound for the problem of super-resolution
of point sources on the grid. In \prettyref{sec:mainsteps} we prove
the main results, and in \prettyref{sec:Numerical-evidence} we present
numerical evidence confirming our bounds.

\section{Preliminaries}
\label{sec:known-bounds}

\subsection{Notation}

\begin{defn}
  For $N\in\NN$ and vector $\xiv=\left(\xi_1,\dots,\xi_s\right)$ of
  pairwise distinct real nodes $\xi_j\in\left(-\pi,\pi\right]$, we define the
  rectangular $\left(2N+1\right)\times s$ Vandermonde matrix
  $\VV_N\left(\xiv\right)$ as
\begin{equation}\label{eq:vand-def-2}
\VV_N\left(\xiv\right):={1\over\sqrt{2N}} \bigl[ \exp\left(\imath k \xi_j\right)\bigr]_{k=-N,\dots,N}^{j=1,\dots,s}.
\end{equation}
\end{defn}

In many applications of interest, the columns of $\VV_N$ as above
arise from sampling the exponential functions
$\left\{\exp\left(\imath\w t_j\right)\right\}_{j=1}^s$ at equispaced
points $\w_k = {k\over N}\O,\;|k|\leq N$, where $\O>0$ is a quantity
which is frequently called the bandlimit or bandwidth, and the nodes
$\{t_j:={N\xi_j\over\O}\}$ represent some relevant physical
parameters, such as angles of arrival, locations of point sources
etc. Therefore, in these cases it is more natural to regard $\{t_j\}$
and $\O$ as the primary variables instead of $\{\xi_j\}$ and $N$,
while in fact thinking about the scenario where $N$ can be very
large. According with this philosophy, we shall be primarily
interested in the continuous limit $N\to\infty$.

\begin{defn}\label{def:vn-rescaled-def}
  For $N\in\NN$, $s\in\NN$, $\xvec$ a vector of $s$ distinct nodes
  $\xvec:=\left(t_1,\dots,t_s\right)$ with
  $t_j\in\left(-{\pi\over 2},{\pi\over 2}\right]$, and bandwidth
  parameter $\O>0$, denote by $\VV_N\left(\xvec,\O\right)$ the
  rectangular $(2N+1)\times s$ Vandermonde matrix with complex nodes
  $z_{j,N}=\exp\left(\imath \xi_{j,N}\right)$ where
  $\xi_{j,N}= \frac{t_j\O}{N}$, i.e.
\begin{equation}
  \label{eq:discretized-vandermonde}
  \VV_{N}\left(\xvec,\O\right):=\VV_N\left({\O\over N}\xvec\right)={1\over\sqrt{2N}}\left[\exp\left(\imath k \frac{t_j\O}{N}\right)\right]_{k=-N,\dots,N}^{j=1,\dots,s}.
\end{equation}
\end{defn}

With the above definition, the Gramian matrix
$\VV_N\left(\xvec,\O\right)^H \VV_N\left(\xvec,\O\right)$ becomes in
the limit $N\to\infty$ the kernel matrix with respect to the
well-known $\sinc$ kernel.

\begin{defn}
For $N\in\NN$, the Dirichlet (periodic sinc) kernel of order $N$ is
$$
\DD_N\left(t\right):=\sum_{k=-N}^N \exp\left(\imath k t\right)=
\begin{cases}
\frac{\sin\left((N+{1\over 2})t\right)}{\sin{t\over 2}} & t \neq 0, \\
2N + 1 & \text{else}.
\end{cases}
$$
\end{defn}

\begin{defn}
  For $N\in\NN$, and $\xvec,\O$ as in \prettyref{def:vn-rescaled-def}, let
  $\GG_N$ be the $s\times s$ matrix
  $$
  \GG_N\left(\xvec,\O\right):=\VV_N\left(\xvec,\O\right)^H
  \VV_N\left(\xvec,\O\right) = {1\over
    {2N}}\left[\DD_N\left({\O\left(t_i-t_j\right)\over
        N}\right)\right]_{i,j}.
$$
\end{defn}

\begin{defn}
Let the $\sinc$ function be defined by
$$
\sinc(t) := {1\over 2} \int_{-1}^1 \exp(\imath\w t) d\w = \begin{cases}
  \frac{\sin t}{t} & t \neq 0, \\
  1 & \text{else}.
\end{cases}
$$
\end{defn}
\begin{defn}\label{def:prolate}
  For $s\in\NN$, $\xvec$ a vector of $s$ distinct nodes
  $\xvec:=\left(t_1,\dots,t_s\right)$ with
  $t_j\in\left(-{\pi\over 2},{\pi\over 2}\right]$, and bandwidth
  parameter $\O>0$, let $\GG\left(\xvec,\O\right)$ denote the
  $s\times s$ matrix
\begin{equation}
  \label{eq:sinc-matrix}
  \GG\left(\xvec,\O\right):=\Biggl[ \sinc \left(\O \left(t_i-t_j\right)\right)\Biggr]_{1\leq i,j\leq s}.
\end{equation}
\end{defn}
\begin{prop}
  For $\xvec$ a vector of pairwise distinct nodes, the matrix
  $\GG\left(\xvec,\O\right)$ is positive definite.
\end{prop}
\begin{proof}
The matrix $\GG$ is nothing but the Gramian matrix of the functions
$\left\{\exp(\imath t_j \w)\right\}_{j=1,\dots,s}$ with the inner
product
$\left\langle f,g\right\rangle_{\O}:={1\over{2\O}} \int_{-\O}^\O
f(\w)\overline{g(\w)}d\w$. For any $\xvec$ as above and nonzero
$\cvec=\left(c_1,\dots,c_s\right) \in \CC^s$ define
$f_{\xvec,\cvec}(\w):=\sum_{j=1}^s c_j \exp(\imath t_j \w) \not\equiv 0$, then we
have
$\|\GG\left(\xvec,\O\right) \cvec \|_2^2 = \left\langle
  f_{\xvec,\cvec},f_{\xvec,\cvec}\right\rangle_{\O} > 0$.  
\end{proof}

For any matrix $\GG \in \CC^{s \times s}$, and a matrix $\VV \in \CC^{N \times s}$ with $N \ge s$, 
we denote as usual
\begin{align*}
  \lambda_{\min}(\GG) &:= \text{The minimal eigenvalue of } \GG;\\ 
  \sigma_{\min}(\VV)  &:= \sqrt{\lambda_{\min}(\VV^H \VV)}.
\end{align*}  

\begin{prop}
  With the above definitions, we have
\begin{equation}\label{eq:sing-eigen-lim}
  \lambda_{\min}\left(\GG\left(\xvec,\O\right)\right)=\lim_{N\to\infty} \lambda_{\min}\left(\GG_N\left(\xvec,\O\right)\right)=\lim_{N\to\infty} \sigma^2_{\min}\left(\VV_N\left(\xvec,\O\right)\right).
\end{equation}
\end{prop}
\begin{proof}
Approximating the integrals by the Riemann sums, we have that
\begin{equation*}
  \sinc\left(\O t\right)=\lim_{N\to\infty}\frac{1}{2N}\sum_{k=-N}^N \exp\left( \imath \frac{k}{N}\Omega t\right)=\lim_{N\to\infty} {1\over{2N}} \DD_N \left({\O t\over N}\right),
\end{equation*}
and therefore
$\GG\left(\xvec,\O\right) = \lim_{N\to\infty}
\GG_N\left(\xvec,\O\right)$. By definition $\VV_N^H\VV_N=\GG_N$, and
so by continuity of eigenvalues \cite[Section 2.4.9]{horn_matrix_2012}
we conclude that \eqref{eq:sing-eigen-lim} holds.
\end{proof}

The main subject of the paper is the scaling of the smallest
eigenvalue of $\GG$ and the smallest singular value of $\VV_N$, when
some of the nodes of $\xvec$ nearly collide (become very close to each
other).

\begin{defn}[Wrap-around distance]
For $t \in {\mathbb R}$, we denote
$$
	\|t\|_{\mathbb{T}}:=\left| \Arg \exp\left(\imath t\right) \right|= \biggl|t \mod \left(-\pi,\pi\right]\biggr|,
$$
where $\Arg(z)$ is the principal value of the argument of $z\in
{\mathbb C} \backslash \{0\}$, taking values in $\left(-\pi,\pi\right]$.
\end{defn}

\begin{defn}[Minimal separation]
  Given a vector of $s$ distinct nodes
  $\xvec:=\left(t_1,\dots,t_s\right)$ with
  $t_j\in\left(-{\pi\over 2},{\pi\over 2}\right]$, we define the
  minimal separation (in the wrap-around sense) as
$$
\Delta=\Delta\left(\xvec\right):=\min_{i\neq j}\|t_i-t_j\|_{\mathbb{T}}.
$$
\end{defn}

\subsection{Known bounds}

Let $\VV_N$ be as defined in \eqref{eq:discretized-vandermonde},
i.e. a rectangular Vandermonde matrix with nodes
$z_{j,N}=\exp\left(\imath \xi_{j,N}\right)$ on the unit circle with
$\xi_{j,N}=t_j {\O\over N}$, $j=1,\ldots,s$. Denote
$\Delta_N:=\min_{i\neq j}|\xi_{i,N}-\xi_{j,N}|$.

Several more or less equivalent bounds on
$\sigma_{\min}\left(\VV_N\right)$ are available in the
``well-separated'' case $N\Delta_N > const$, using various results
from analysis and number theory such as Ingham and Hilbert
inequalities, large sieve inequalities and Selberg's majorants
\cite{ingham_trigonometrical_1936,moitra_super-resolution_2015,negreanu_discrete_2006,aubel_vandermonde_2017,montgomery_ten_1994,montgomery_hilberts_1974,ferreira_superresolution_1999,bazan_conditioning_2000}.

The tightest bound was obtained in \cite{aubel_vandermonde_2017}
(slightly improving Moitra's bound from
\cite{moitra_super-resolution_2015}), where it was shown that (in our
notations we substitute $N \rightarrow 2N+1$) if
$2N+1 > {2\pi\over \Delta_{N}}$ then
$$
\sigma_{\min}\left(\sqrt{2N}\VV_N\right) \geq \sqrt{2N+1-{2\pi\over\Delta_N}}.
$$

In our setting, we have $\Delta_N={\Delta\O\over N}$ and so as
$N\to\infty$ we obtain, assuming $\Delta\O \geq \pi$, that
\begin{equation*}
  \sigma_{\min}\left(\VV_N\right) \geq \sqrt{1+{1\over {2N}}-{2\pi\over {2N\Delta_N}}} \to \sqrt{1-{\pi\over\O\Delta}}.
\end{equation*}

The case $\Delta \O \ll 1$, or, equivalently,
$\min_{i\neq j}|\xi_{i,N}-\xi_{j,N}|\ll{1\over N}$, turns out to be much
more difficult to analyze. All known results provide sharp bounds only
in the particular case when all the nodes are clustered together, or
approximately equispaced.

If all the nodes $t_j$ are equispaced, say $t_j=t_0+j\Delta,\;j=1,\dots,s$, then the
matrix $\GG$ is the so-called \emph{prolate matrix}, whose spectral
properties are known exactly
\cite{varah_prolate_1993,slepian_prolate_1978}. Indeed, we have in
this case
$$
\GG_{i,j}=\frac{\sin\left(\O\left(t_i-t_j\right)\right)}{\O(t_i-t_j)}=\frac{\sin\left(\O\Delta\left(i-j\right)\right)}{\O\Delta\left(i-j\right)}=\frac{\pi}{\O\Delta} \cdot \frac{\sin\left(2\pi W\left(i-j\right)\right)}{\pi\left(i-j\right)},\quad W:=\frac{\O\Delta}{2\pi},
$$
and therefore $\GG=\frac{\pi}{\O\Delta}\boldsymbol{Q}(s,W)$ where
$\boldsymbol{Q}(s,W)$ is the matrix defined in
\cite[eq. (21)]{slepian_prolate_1978}. The smallest eigenvalue of
$\boldsymbol{Q}(s,W)$, denoted by $\lambda_{s-1}(s,W)$ in the same paper, has
the exact asymptotics for $W$ small, given in
\cite[eqs. (64,65)]{slepian_prolate_1978}:
\begin{equation}\label{eq:slepian-explicit-bound}
  \lambda_{s-1}\left(s,W\right)={1\over\pi}\left(2\pi W\right)^{2s-1} \Cl{slepian}(s) \left(1+O\left(W\right)\right),\quad \Cr{slepian}(s):=\frac{2^{2s-2}}{\left(2s-1\right){{2s-2} \choose {s-1}}^3},
\end{equation}
which gives
$$
\lambda_{\min}\left(\GG\right)=\Cr{slepian}\left(s\right)\left(\O\Delta\right)^{2s-2} \left(1+O\left(\O\Delta\right)\right),\quad \O\Delta\ll 1.
$$

The same scaling was shown using Szego's theory of Toeplitz forms in
\cite{demanet_recoverability_2014} -- see also
\prettyref{sub:discussion}. The authors showed that there exist $C>0$
and $y^*>0$ such that for $\O\Delta<y^*$
$$
{C\over 16} \left(\sin{2\O\Delta\over \pi}\right)^{2s-2} \leq \lambda_{\min}\left(\GG\right)\leq 16 \left(\sin{2\O\Delta\over \pi}\right)^{2s-2}.
$$

To conclude the above discussion, defining the ``super-resolution factor'' as
$$
\srf:={\pi\over \Delta \O},
$$
we have that
\begin{align}
  \label{eq:separated-bound}
  \lambda_{\min}\left(\GG\right) & \approx \left(1-\srf\right),\quad \srf \leq 1; \\
    \label{eq:slepian-bound}
  \lambda_{\min}\left(\GG\right) &\approx \srf^{-2(s-1)},\quad\srf \gg 1.
\end{align}

\section{Main results}
\label{sec:main-results}

\subsection{Optimal bounds for the smallest eigenvalue}
\label{sec:main-result-intro}

It turns out that the bound \eqref{eq:slepian-bound} is too
pessimistic if only some of the nodes are known to be
clustered. Consider for instance the configuration
$\xvec=\left(t_1=\Delta,\; t_2=2\Delta,\;t_3= -\frac{\pi}{4}\right)$,
then, as can be seen in \prettyref{fig:sigma.min.first.simulation}, we
have in fact
$\lambda_{\min}\left(\GG\left(\xvec,\O\right)\right) \approx
\left(\Delta \O\right)^2$, decaying much slower than $(\Delta \O)^4$
-- which would be the bound given by \eqref{eq:slepian-bound}.

\begin{figure}[t]
  \centering \subfloat[Schematic representation of
  $\xvec$.]{\includegraphics[width=0.45\linewidth]{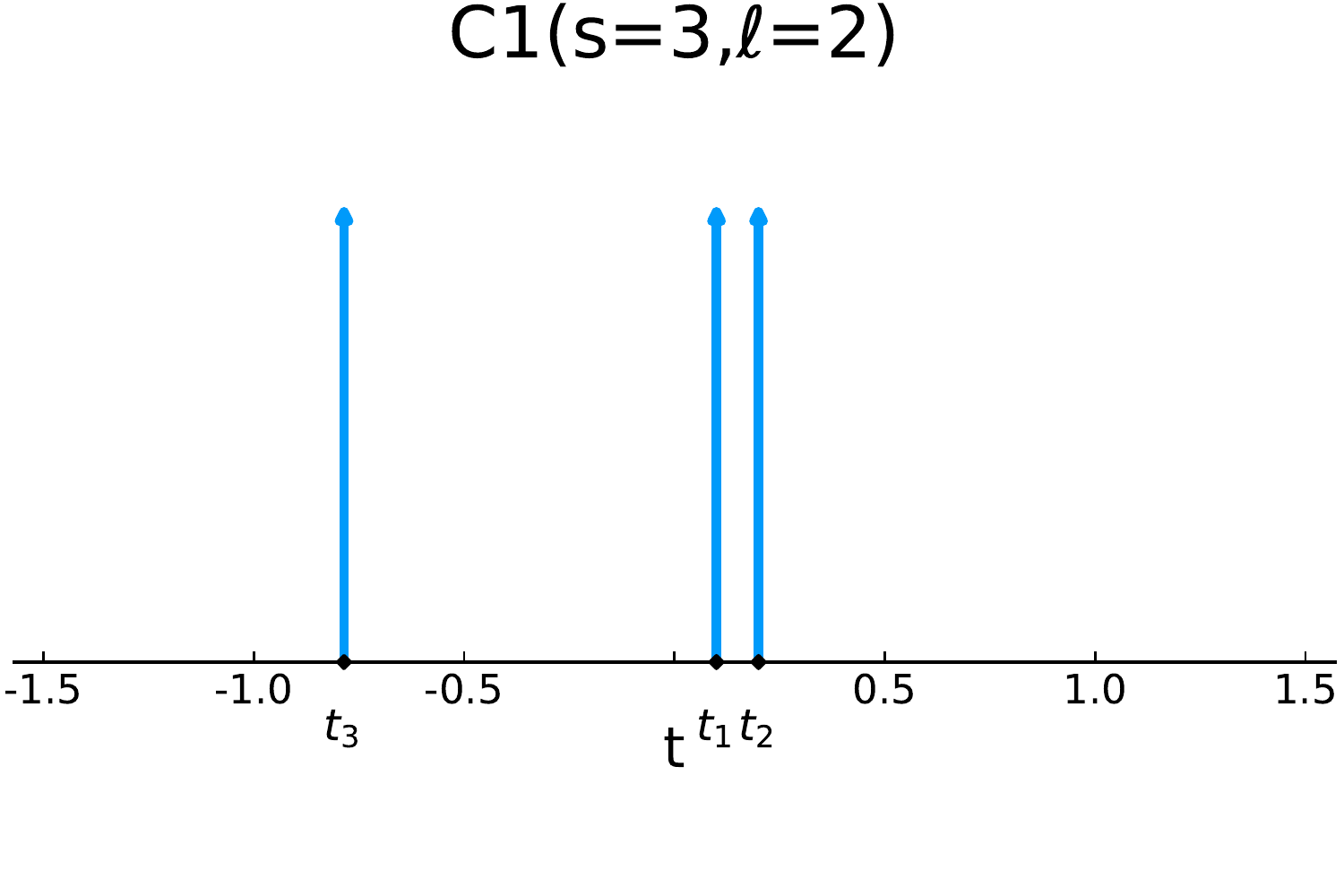}\label{subfig:demo1-conf}}
  \subfloat[The decay of
  $\lambda_{\min}$.]{\includegraphics[width=0.45\linewidth]{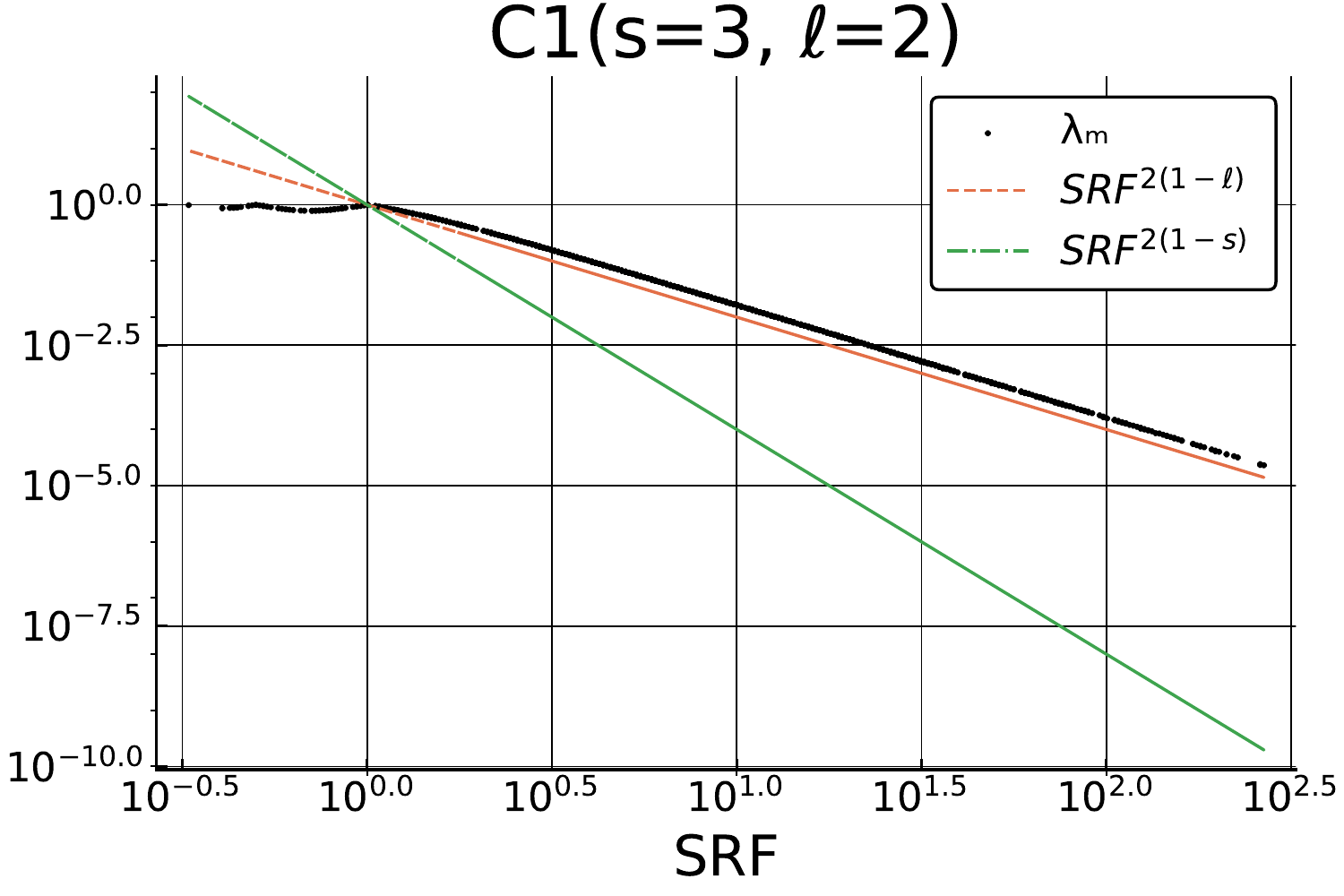}\label{subfig:demo1}}
  \caption{\small For different values of $\Delta,\O$ we plot the
    quantity
    $\lambda_{m}=\lambda_{\min}\left(\GG\left(\xvec,\O\right)\right)$
    versus the super-resolution factor $\srf=\frac{\pi}{\Delta
      \O}$. {\bf (a)}
    $\xvec=\left(t_1=\Delta,\; t_2=2\Delta,\;t_3=
      -\frac{\pi}{4}\right)$ is a single cluster with $s=3$ and
    $\ell=2$. {\bf (b)} The correct scaling is seen to be
    $\lambda_{m}\sim \left(\Delta\O\right)^{2\left(\ell-1\right)}$
    rather than
    $\lambda_{m}\sim \left(\Delta\O\right)^{2\left(s-1\right)}$.  See
    \prettyref{sec:Numerical-evidence} for further details regarding
    the experimental setup. The relationship breaks when
    $\srf\leq O(1)$, consistent with \eqref{eq:separated-bound}. }
 \label{fig:sigma.min.first.simulation}
\end{figure}

In this paper we bridge this theoretical gap. We consider the
\emph{partially clustered regime} where at most $2\leq\ell\leq s$
neighboring nodes can form a cluster (there can be several such
clusters), with two additional parameters $\rho,\tau,$ controlling the
distance between the clusters and the uniformity of the distribution
of nodes within the clusters.

\begin{figure}
  \centering
  \includegraphics[width=\linewidth]{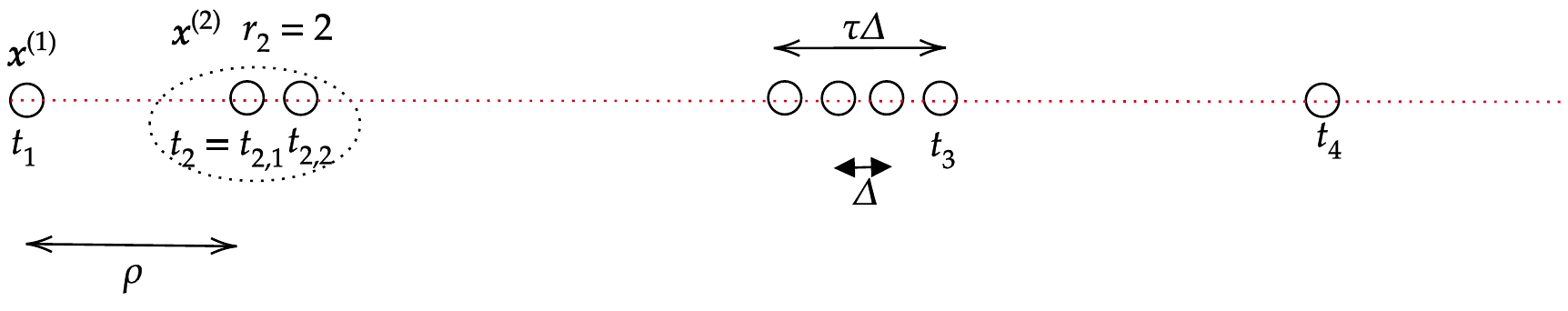}
  \caption{\small The schematic representation of a cluster configuration
    according to \prettyref{def:partial-cluster}. Here $s=8$ and
    $\ell=4$. Each node $t_j$ defines its ``cluster'' $\xvec^{(j)}$ of
    size $r_j \leq \ell$. $\rho$ is the minimal distance from any node
    $t_j$ to another node $y$ not in $\xvec^{(j)}$. The distance
    between any two nodes in $\xvec^{(j)}$ is between $\Delta$ and
    $\tau\Delta$.}
  \label{fig:clusters-conf-def}
\end{figure}

\begin{defn}
  \label{def:partial-cluster}
  The node vector
  $\xvec=\left(t_1,\dots,t_s\right) \subset
  (-\frac{\pi}{2},\frac{\pi}{2}]$ is said to form a
  $\left(\Delta,\rho,s,\ell,\tau\right)$-clustered configuration for
  some $\Delta>0$, $2\leq \ell\leq s$, $\ell-1\leq\tau<{\pi\over\Delta}$ and
  $\rho \ge 0$, if for each $t_j$, there exist at most $\ell$ distinct
  nodes
  $$
  \xvec^{(j)}=\{t_{j,k}\}_{k=1,\dots,r_j} \subset \xvec,\;1\leq
  r_j\leq\ell,\quad t_{j,1}\equiv t_j,
  $$
  such that the following conditions are satisfied:
  \begin{enumerate}
  \item For any $y\in\xvec^{(j)}\setminus\{t_j\}$, we have
    $$
    \Delta\leq \|y-t_j\|_{\mathbb{T}} \leq \tau \Delta.
    $$
  \item For any $y\in\xvec\setminus\xvec^{(j)}$, we have
    $$
    \|y-t_j\|_{\mathbb{T}} \geq \rho. 
    $$
  \end{enumerate}
\end{defn}

The different parameters are illustrated in
\prettyref{fig:clusters-conf-def}.

Our main result is the following generalization of
\eqref{eq:slepian-bound} for clustered configurations.
\begin{thm}
  \label{thm:main-theorem}
  There exists a constant $\Cl{main-c-n}=\Cr{main-c-n}\left(s\right)$
  such that for any $4\tau \Delta \le \rho$, any $\xvec$ forming a
  $\left(\Delta,\rho,s,\ell,\tau\right)$-clustered configuration, and
  any $\O$ satisfying
  \begin{equation}\label{eq:Omega-req-th1}
    \frac{4\pi s}{\rho} \le \O \le{\pi s\over{\tau\Delta}},
  \end{equation}
  we have
  \begin{align}
      \label{eq:main-bound-finite-n-g}
      \sigma_{\min}\left(\VV_N\left(\xvec,\O\right)\right) &\geq
                                                             \Cr{main-c-n} \cdot \left(\Delta \O\right)^{\ell-1},\qquad \text{ whenever } N > 2s^3 \left\lceil \frac{\O}{4s}\right\rceil;\\
      \label{eq:main-bound-thm}
      \lambda_{\min}\left(\GG\left(\xvec,\O\right)\right) &\geq \Cr{main-c-n}^2 \cdot \left(\Delta \O\right)^{2\left(\ell-1\right)}.
    \end{align}
\end{thm}

The proof of \prettyref{thm:main-theorem} is presented in
\prettyref{sub:theproof} below. It is based on the
``decimation-and-blowup'' technique, previously used in the context of
super-resolution in
\cite{akinshin_accuracy_2015,akinshin2017geometry,batenkov_accurate_2017,batenkov_stability_2016,superres_clusters18}
and references therein. In a nutshell, the main idea is to choose an
appropriate ``decimation'' parameter $\lambda \approx \O$ such that
the ``inflated'' nodes in the vector $\lambda\xvec$ (considered in the
wrap-around sense) are separated by $\lambda\Delta \approx \O\Delta$
from its cluster neighbors, and by a constant from the other
nodes. Then we fix sufficiently large $N$ and divide the $2N+1$ rows
of $\VV_N$ into groups of $s$ rows, separated by $\lambda
N\over\O$. Each of the resulting \emph{square} Vandermonde matrices
can be explicitly estimated (the inverses have well-known behaviour),
and has smallest singular value of the order
${1\over\sqrt{N}} \left(\Delta\O\right)^{\ell-1}$. The main technical
part is to show that such $\lambda$ exists, and it is proved in
\prettyref{lem:blowup-lemma} by a union bound argument, showing that
the measure of all ``bad'' values of $\lambda$ (causing a collision of
at least two nodes) is small. The condition on $N$ in
\eqref{eq:main-bound-finite-n-g} is obtained by accurate counting of
how many such ``bad'' intervals exist.

\begin{rem}
  The condition $4\tau\Delta \leq \rho$ ensures that the range of
  admissible $\O$ is non-empty, and it will clearly be satisfied for
  all small enough $\Delta$ with all the rest of the parameters fixed.
\end{rem}

\begin{rem}\label{rem:different-confs}
  The same node vector $\xvec$ can be regarded as a clustered
  configuration with different choices of the parameters
  $\left(\ell,\rho,\tau\right)$. For example, the vector $\xvec$ from
  the beginning of this section (and also
  \prettyref{fig:sigma.min.first.simulation}) is both
  $\left(\Delta,{\pi\over 4}+\Delta,3,2,1\right)$-clustered and
  $\left(\Delta,\rho,3,3,{\pi\over{4\Delta}} +2 \right)$-clustered, with
  any $\rho$. To obtain as tight a bound as possible, one should
  choose the minimal $\ell$ such that the condition
  \eqref{eq:Omega-req-th1} is satisfied for $\Omega$ within the
  range of interest. For instance, $\Omega$ might be too small if
  $\rho$ is small enough, however by choosing $\ell=s$ one is able to
  increase $\rho$ without bound. See \prettyref{fig:breakdown} for a
  numerical example.
\end{rem}

\begin{rem}
  The constant $\Cr{main-c-n}$ is given explicitly in
  \eqref{eq:main-c-n-def}, and it decays in $s$
  like $\sim s^{-2s}$. It is
  plausible that the best possible bound would scale like $c^{-\ell}$ for
  some absolute constant $c > 1$, see also
  \prettyref{rem:compare-with-LL} below.
\end{rem}

Our next result is the analogue of \eqref{eq:main-bound-finite-n-g}
for the Vandermonde matrix $\VV_N$ as in \eqref{eq:vand-def-2}, albeit
under an extra assumption that the nodes are restricted to the
interval $\frac{1}{s^2} \left(-\frac{\pi}{2},\frac{\pi}{2}\right]$.

\begin{cor}\label{cor:finite-n-lower-bound}
  There exists a constant
  $\Cl{main-vand}=\Cr{main-vand}\left(s\right)$ such that for any
  $4\tau\Delta\le\min\left({\rho}, {1\over{s^2}}\right)$, any
  $\xiv=\left(\xi_1,\dots,\xi_s\right) \subset \frac{1}{s^2}
  \left(-\frac{\pi}{2},\frac{\pi}{2}\right]$ forming a
  $\left(\Delta,\rho,s,\ell,\tau\right)$-clustered configuration, and
  any $N$ satisfying
  \begin{equation}
    \label{eq:n-cond-cor1}
    \max\left(\frac{4\pi s}{\rho},4s^3\right) \leq N \leq {\pi
      s\over{\tau\Delta}},
  \end{equation}
  we have
  \begin{equation}\label{eq:sigma-min-vand-finite-n}
    \sigma_{\min}\left(\VV_N\left(\xiv\right)\right)\geq
    \Cr{main-vand}\cdot \left(N\Delta\right)^{\ell-1}.
  \end{equation}
\end{cor}

\begin{proof}
  \newcommand{\tO}{\widetilde{\O}} \renewcommand{\tt}{\widetilde{t}}
  \newcommand{\tD}{\widetilde{\Delta}}
  \newcommand{\tX}{\widetilde{\xvec}}
  \newcommand{\tE}{\widetilde{\eta}}
  \newcommand{\tR}{\widetilde{\rho}}

  Let us choose $\tO:={N\over s^2}$ so that for all $j=1,\dots,s$ we
  have
  \begin{equation*}
    \tt_j:={N\xi_j\over \tO} \in\left(-{\pi\over 2},{\pi\over 2}\right].
  \end{equation*}

  Further define $\tD:=s^2 \Delta$, and $\tR:=s^2\rho$. We immediately
  obtain that the vector $\tX:=\left(\tt_1,\dots,\tt_s\right)$ forms a
  $\left(\tD,\tilde{\rho},s,\ell,\tau\right)$-clustered configuration
  according to \prettyref{def:partial-cluster}, and the rectangular
  Vandermonde matrix $\VV_N\left(\xiv\right)$ in \eqref{eq:vand-def-2}
  is precisely $\VV_N\left(\tX,\tO\right)$. Clearly,
  $4\tau\tD \le s^2 \rho = \tR$, and also
  \begin{equation}\label{eq:cond-nsatisfied-cor1}
    \tO s^2 = N \geq 4s^3 \Longrightarrow {\tO \over {4s}} \geq 1
    \Longrightarrow {2\tO \over {4s}} > \left\lceil {\tO \over
        {4s}}\right\rceil \Longrightarrow N = \tO s^2 > 2s^3 \left\lceil
      {\tO \over {4s}}\right\rceil.
  \end{equation}
  Using \eqref{eq:n-cond-cor1}, we obtain precisely the conditions
  \eqref{eq:Omega-req-th1} with $\tO,\tR$ in place of $\Omega,\rho$
  respectively. Therefore the conditions of
  \prettyref{thm:main-theorem} are satisfied for
  $\tX,\tO,\tR,\tD,\tau$, and so \eqref{eq:sigma-min-vand-finite-n}
  follows immediately from \eqref{eq:cond-nsatisfied-cor1} and
  \eqref{eq:main-bound-finite-n-g}, with
  $\Cr{main-vand}=\Cr{main-c-n}$.
\end{proof}

\begin{rem}\label{rem:compare-with-LL}
  During the revision of the present paper, the authors of
  \cite{li_stable_2017} (second version) investigated the question of
  bounding $\sigma_{\min}(\VV_N)$ under assumptions on node
  distribution which are similar to our clustering model (they are
  called ``sparse clumps'' in \cite{li_stable_2017}.) They also obtain
  the scaling $\left(N\Delta\right)^{\ell-1}$ for the smallest
  singular value. Comparing their results to
  \prettyref{cor:finite-n-lower-bound} (see also Remark 4 in their
  paper), we note the following.
  \begin{enumerate}
  \item They do not have the requirement that the vector $\xiv$ should
    be restricted to a small interval.
  \item Their bounds hold whenever $N\geq s^2$, while we require
    $N\geq 4s^3$. 
  \item Although their model is more general, their constants are more
    complicated. Nevertheless, the corresponding constant
    $\Cr{main-vand}$ scales as $\ell^{-\ell}$ which is better than our
    $s^{-2s}$.
  \item Their equation (2.5) in Theorem 2 requires the
    product $\rho N$ to be at least $\ell^{5/2}{20s \over{\sqrt{N\Delta}}}$,
    which essentially forces a single cluster if $\Delta$ is very
    small (or, alternatively, prevents $\Delta$ to be too small for
    certain $s,\ell$) \footnote{The particular
      equation and theorem number might change as
      \cite{li_stable_2017} is currently a preprint.}. In contrast, our equation
    \eqref{eq:n-cond-cor1} only requires $\rho N \geq 4\pi s$, and
    therefore doesn't have these restrictions (although both
    conditions require $\rho$ to grow with $s$.)
  \end{enumerate}
\end{rem}

\begin{rem}
  Continuing the above discussion, we would like to emphasize that
  \prettyref{cor:finite-n-lower-bound} is derived by discretization of
  the continuous setting of \prettyref{thm:main-theorem}, and
  therefore it is perhaps not surprising that the conditions for which
  the scaling holds are not optimal.
\end{rem}

Returning back to \prettyref{thm:main-theorem}, it turns out that the
bound \eqref{eq:main-bound-thm} is asymptotically optimal.

\begin{thm}\label{thm:optimality}
  There exists an absolute constant $\eta\ll 1$ and a constant
  $\Cl{upper}=\Cr{upper}\left(\ell\right)$ such that for any
  $2\leq \ell\leq s$ and any $\Delta$ satisfying
  $\Delta<{\pi\over{2(\ell-1)}}$, there exists a
  $\left(\Delta,\rho',s,\ell,\tau'\right)$-clustered configuration
  $\xvec_{\min}$ with $s$ nodes and certain $\rho',\tau'$ depending
  only on $s,\ell$, for which
  $$
  \lambda_{\min}\left(\GG\left(\xvec_{\min},\O\right)\right) \leq
  \Cr{upper}\cdot \left(\Delta\O\right)^{2\left(\ell-1\right)},\qquad  \Delta\O<\eta.
  $$
\end{thm}

The proof of \prettyref{thm:optimality} is presented in
\prettyref{sub:optimality}.

Finally we conclude with the optimal scaling for the
\emph{condition number} of $\VV_N=\VV_N\left(\xvec,\O\right)$, which
is of interest to some applications.
\begin{cor}\label{cor:cond-number}
  Fix $s,\ell,\rho,\tau$ and $\O$. As $\Delta\to 0$ and $N\to\infty$,
  for any $\left(\Delta,\rho,s,\ell,\tau\right)$-clustered
  configuration $\xvec$, we have
$$
\kappa\left(\VV_N\left(\xvec,\O\right)\right):=\frac{\sigma_{\max}(\VV_N)}{\sigma_{\min}(\VV_N)}
\asymp \srf^{\ell-1}.
$$
\end{cor}
\begin{proof}
  It is immediate that as $N\to\infty$, the largest singular value
  (the spectral norm) of $\VV_N$ is bounded from  above by a
  constant:
  $$
  \sigma_{\max}(\VV_N) =  \|\VV_N\|_2 \leq \sqrt{s{{2N+1}\over{2N}}} \leq \sqrt{2s},
  $$
  while the lower bound can be obtained by
  $$
  \sigma_{max}(\VV_N) = \sqrt{\lambda_{\max}(\GG_N)} \geq \sqrt{{1\over{2N}}\max_{t\in\RR} \DD_N\left(t\right)} > 1.
  $$
  Combining this with \prettyref{thm:main-theorem} and
  \prettyref{thm:optimality} finishes the proof.
\end{proof}

\subsection{Stable super-resolution of point
  sources}\label{sub:discussion}

The problem of (sparse) super-resolution is to recover discrete,
point-like objects from their noisy and bandlimited spectral
measurements. It arises in many fields such as frequency estimation,
sampling theory, array processing, astronomical imaging, seismic
imaging, nonuniform FFT, statistics, radar signal detection, error
correction codes, and others
\cite{auton_investigation_1981,candes_towards_2014,cuyt_sparse_2016,donoho_superresolution_1992,blu2008,fomel_seismic_2013,li_phase_2015,li_parametric_2000,pan_towards_2016}. Our
main results have direct implications for the problem of
super-resolution under sparsity constraints, in the so-called
``on-grid'' model\footnote{Note that the results in the previous
  section are valid for ``off-grid'' setting, as the nodes $\{t_j\}$
  can have arbitrary real values in
  $\left(-{\pi\over 2},{\pi\over 2}\right]$.}.

\begin{defn}\label{def:grid}
For $\Delta>0$, denote by $\grd$ the discrete grid
$$
\grd:=\left\{k\Delta,
  \;k=-\left\lfloor{\pi\over{2\Delta}}\right\rfloor,\dots,\left\lfloor{\pi\over{2\Delta}}\right\rfloor\right\} \subset \left[-{\pi\over 2},{\pi\over 2}\right].
$$  
\end{defn}

\begin{defn}\label{def:R-class}
  For $\Delta,\rho,s,\ell,\tau$ as in \prettyref{def:partial-cluster},
  let $\grdM:=\grdM\left(\Delta,\rho,s,\ell,\tau\right)$ be the set of
  point measures of the form $\mu=\sum_{j=1}^s a_j \delta_{t_j}$,
  where $t_j\in\grd$ for all $j=1,\dots,s$, $\delta_t$ is the Dirac
  measure supported on $t\in\RR$, $a_j\in\CC$, and the node vector
  $\left(t_1,\dots,t_s\right)$ forms a
  $\left(\Delta,\rho,s,\ell,\tau\right)$-clustered configuration
  according to \prettyref{def:partial-cluster}.
\end{defn}

Consider the problem of reconstructing $\mu\in\grdM$ from approximate
spectral data $\widehat{\mu}\left(\w\right)$ restricted to some
interval $\w\in\left[-\O,\O\right]$. Here the Fourier transform
$\widehat{\mu}$ is defined as
$$
\mu = \sum_{j=1}^s a_j \delta_{t_j} \Longrightarrow \widehat{\mu}\left(\w\right) = \sum_{j=1}^s a_j \exp\left(\imath \w t_j\right).
$$

The measurement space $L_2\left(\left[-\O,\O\right]\right)$ contains
complex-valued square-integrable functions supported on
$\left[-\O,\O\right]$, with the norm
\begin{equation}\label{eq:omega-l2-norm}
\|f\|^2_{2,\O} := {1\over{2\O}}\int_{-\O}^{\O} |f\left(\w\right)|^2 d\w.
\end{equation}

Proceeding as in
\cite{donoho_superresolution_1992,demanet_recoverability_2014}, we
define the minimax error for this problem as follows.

\begin{defn}\label{def:minimax-error}
  For $\grdM$ as above, $\varepsilon>0$ and $\Omega > 0$, the minimax error
  $\minmax=\minmax(\grdM,\Omega,\varepsilon)$ is the quantity
  \begin{equation}
    \label{eq:minmax-def}
    \minmax := \inf_{\widetilde{\mu}\left(\Phi_{\mu,e}\right)\in\grdM}\sup_{\mu\in\grdM}\sup_{e \in
      L_2\left(\left[-\Omega,\Omega\right]\right),\;\|e\|_{2,\O}\leq\varepsilon} \|\widetilde{\mu}-\mu\|_2,
  \end{equation}
  where
  \begin{itemize}
  \item
    $\Phi_{\mu,e} \in L_2\left(\left[-\Omega,\Omega\right]\right)$
    is the measurement function given by
    \begin{equation}\label{eq:measurement-fun}
    \Phi_{\mu,e}(\omega) = \widehat{\mu}(\w) + e(\w);
  \end{equation}
  
  \item $\widetilde{\mu}$ is any
    deterministic mapping from
    $L_2\left(\left[-\Omega,\Omega\right]\right)$ to $\grdM$;
  \item for $\mu=\sum_{j=1}^s a_j \delta_{t_j}$, the norm $\|\mu\|_2$
    is the discrete $\ell_2$ norm of the coefficient vector:
    $$
    \|\mu\|_2:=\left(\sum_{j=1}^s |a_j|^2\right)^{1\over 2}.
    $$
  \end{itemize}
\end{defn}

Using arguments very similar to
\cite{morgenshtern2016,donoho_superresolution_1992,demanet_recoverability_2014,li_stable_2017}
and the novel bounds of \prettyref{thm:main-theorem} and
\prettyref{thm:optimality}, we obtain the optimal rate for the minimax
error for clustered on-grid super-resolution.

\begin{thm}\label{thm:minimax}
  Fix $s \geq 1,\;2\leq \ell \leq s,\;\varepsilon>0$. Put
  $\srf:={\pi\over{\Delta\Omega}}$.  Then the following hold.
  \begin{enumerate}
  \item For any $\rho\geq 0,\;\ell-1\leq \tau$ and $M\geq \pi$, there
    exists $\alpha \geq M$ such that for all
    sufficiently small $\Delta$ it holds that
  \begin{equation}
  \label{eq:new-minimax-upper}
    \minmax\left(\grdM\left(\Delta,\rho,s,\ell,\tau\right),\O,\varepsilon\right) \leq C_{s,\ell} \srf^{2\ell-1} \varepsilon,\quad\srf=\alpha,
  \end{equation}
  for some absolute constant $C_{s,\ell}$ depending only on $s$ and $\ell$.
\item There exists an absolute constant $\beta \gg 1$ and
  $\rho',\tau'$, depending only on $s,\ell$, such that for any
  $\Delta<{\pi\over{2(2\ell-1)}}$ it holds that
  \begin{equation}
    \label{eq:new-minimax-lower}
    \minmax\left(\grdM\left(\Delta,\rho',s,\ell,\tau'\right),\O,\varepsilon\right) \geq C_{\ell} \srf^{2\ell-1}\varepsilon,\quad \srf > \beta,
  \end{equation}
  for some absolute constant $C_{\ell}$ depending only on $\ell$.
  \end{enumerate}
\end{thm}

For the proof, see \prettyref{sub:proof-minimax} below. This result
generalizes \cite{demanet_recoverability_2014,li_stable_2017} (where
the scaling $\minmax \asymp \srf^{2\ell-1}\varepsilon$ was derived for
$\ell=s$), as well as \cite{morgenshtern2016} (where it was shown that
for positive $a_j$ it holds that
$\minmax \lessapprox \srf^{2\ell}\varepsilon$, with a comparable
definition of the Rayleigh regularity $\ell$).

A different but closely related setting was considered in the seminal
paper \cite{donoho_superresolution_1992}, where the measure $\mu$ was
assumed to have infinite number of spikes on a grid of size $\Delta$,
with one spike per unit of time on average, but whose local complexity
was constrained to have not more than $R$ spikes per any interval of
length $R$ (such $R$ is called the ``Rayleigh index''). It was shown
in \cite{donoho_superresolution_1992} that the minimax recovery rate
for such measures scales like $\srf^{\alpha}$ where
$2R-1 \leq \alpha \leq 2R+1$. Our partial cluster model can therefore
be regarded as the finite-dimensional version of these ``sparsely
clumped'' measures with finite Rayleigh index, showing the same
scaling of the error -- polynomial in $\srf$ and exponential in the
``local complexity'' of the signal.

If the grid assumption is relaxed, then one might wish to measure the
accuracy of recovery $\|\widetilde{\mu}-\mu\|$ by comparing the
locations of the recovered signal $\widetilde{\mu}$ with the true ones
$\{t_j\}$. In this case, there are additional considerations which are
required to derive the minimax rate, and it is possible to do so under
the partial clustering assumptions. See
\cite{akinshin_accuracy_2015,superres_clusters18} for details, where
we prove that $\minmax \asymp \srf^{2\ell-1}\Delta \varepsilon$ in this
scenario, for uniform bound on the noise
$\|e\|_{\infty}:=\sup_{|\omega|\leq\Omega}\left|e\left(\omega\right)\right|\leq
\varepsilon \lessapprox \srf^{1-2\ell}$. The extreme case $\ell=s$ has
been treated recently in
\cite{batenkov_accurate_2017,batenkov_stability_2016}.

In the case of well-separated spikes (i.e. clusters of size $\ell=1$),
a recent line of work using $\ell_1$ minimization
(\cite{candes_towards_2014,candes_super-resolution_2013,duval_exact_2014,de_castro_exact_2012}
and the great number of follow-up papers) has shown that the problem is
stable and tractable.

Therefore, the partial clustering case is somewhat mid-way between the
extremes $\ell=1$ and $\ell=s$, and while our results in this paper
(and also in \cite{superres_clusters18}) show that it
is much more stable than in the unconstrained sparse case, it is an
intriguing open question whether provably tractable solution
algorithms exist.

Several candidate algorithms for sparse super-resolution are
well-known -- MUSIC, ESPRIT/matrix pencil, and variants; these have
roots in parametric spectral estimation
\cite{stoica_spectral_2005}. In recent years, the super-resolution
properties of these algorithms are a subject of ongoing interest, see
e.g. \cite{fannjiang_compressive_2016,liao_music_2016,bini_error_2017,li_stable_2017,li2019}
and references therein. Smallest singular values of the partial
Fourier matrices $\VV_N$, for finite $N$, play a major role in these
works, and therefore we hope that our results and techniques may be
extended to analyze these algorithms as well.

\section{Proofs}
\label{sec:mainsteps}

\subsection{Blowup}

Here we introduce the uniform blowup of a node vector
$\xvec=\left(t_1,\dots,t_s\right)$ by a positive parameter $\lambda$,
and study the effect of such a blowup mapping on the minimal
wrap-around distance between the mapped nodes.

\begin{lem}\label{lem:blowup-lemma}
  Let $\xvec$ form a $\left(\Delta,\rho,s,\ell,\tau\right)$ cluster, and
  suppose that $\frac{4\pi s}{\rho} \le \O \le{\pi s\over{\tau\Delta}}$. Then, for any
  $0\leq\xi\leq 1$ there exists a set
  $I\subset\left[{\O \over 2s}, {\O\over s} \right]$ of
  total measure ${\O \over 2s}\xi$ such that for every
  $\lambda\in I$ the following holds for every $t_j\in\xvec$:
  \begin{align}
    \label{eq:blown-up1}
    \|\lambda y -\lambda t_j\|_{\mathbb{T}} & \geq \lambda\Delta\geq{\Delta \O \over {2s}}, &&\forall y\in\xvec^{(j)}\setminus\{t_j\}; \\
    \label{eq:blown-up2}
    \|\lambda y -\lambda t_j\|_{\mathbb{T}} & \geq \frac{1-\xi}{s^2}\pi, &&\forall y\in\xvec\setminus\xvec^{(j)}.
  \end{align}

  Furthermore, the set   $I^c:=\left[{\O\over {2s}},{\O\over
      s}\right]\setminus I$ is a union of at most ${s^2\over 2} \left\lceil{\O\over {4 s}}\right\rceil$ intervals.
\end{lem}

\newcommand{\MP}{\gamma}

\begin{proof}
  We begin with \eqref{eq:blown-up1}. Let
  $\lambda\in\left[{\O \over 2s}, {\O\over s} \right]$, then
  $\lambda\tau\Delta\leq\pi$ and since
  $\|t_j-y\|_{\mathbb{T}}\leq\tau\Delta$ we immediately conclude that
  $$
  \|\lambda t_j-\lambda
  y\|_{\mathbb{T}}=\lambda\|t_j-y\|_{\mathbb{T}}\geq\lambda\Delta.
  $$

  To show \eqref{eq:blown-up2}, let $\nu$ be the uniform probability
  measure on $\left[{\O \over 2s}, {\O\over s} \right]$.
  Let $t_j\in\xvec$ and $y\in\xvec\setminus\xvec^{(j)}$ be
  fixed and put $\delta:=\|y-t_j\|_{\mathbb{T}}$. For $\lambda \in \left[{\O \over 2s}, {\O\over s} \right]$,
  let $\MP(\lambda)=\MP^{(t_j,y)}(\lambda)$ be the
  random variable on $\nu$, defined by
  $$
  \MP^{(t_j,y)}(\lambda):=\|\lambda t_j-\lambda y\|_{\mathbb{T}}.
  $$

  We now show that for any $0\le\alpha\le 1$
  \begin{equation}
    \label{eq:measure-argument-single}
    \nu\left\{\MP\left(\lambda\right) \leq \alpha\pi\right\} \leq 2\alpha.
  \end{equation}
  Since $\delta\geq\rho\geq{4\pi s\over\O}$, we can write
  ${\O\over{2s}}={2\pi\over\delta}\left(n+\zeta\right)$ where
  $n\geq 1$ is an integer and $0\leq\zeta<1$. We break up the
  probability \eqref{eq:measure-argument-single} as follows:
  \begin{align}
    \begin{split}\label{eq:split-probability}
      \nu\left\{\MP\left(\lambda\right) \leq \alpha\pi\right\} &= \sum_{k=1}^n\nu\left\{\MP\left(\lambda\right) \leq \alpha\pi \Biggl| \lambda-{\O\over{2s}}\in{2\pi\over\delta}\left[k-1,k\right] \right\} \nu\left\{\lambda-{\O\over{2s}}\in{2\pi\over\delta}\left[k-1,k\right]\right\}\\
      & \qquad + \nu\left\{\MP\left(\lambda\right) \leq \alpha\pi \Biggl| \lambda-{\O\over{2s}}\in{2\pi\over\delta}\left[n,n+\zeta\right] \right\} \nu\left\{\lambda-{\O\over{2s}}\in{2\pi\over\delta}\left[n,n+\zeta\right]\right\}.
    \end{split}
  \end{align}
  Now, consider the number $a=y-t_j$. As $\lambda$ varies between
  ${\O\over{2s}}+{2(k-1)\pi\over\delta}$ and
  ${\O\over{2s}}+{2k\pi\over\delta}$, the number $\exp(\imath
  \lambda a)$ traverses the unit circle exactly once, and therefore
  the variable $\MP(\lambda)$ traverses the interval $[0,\alpha\pi]$
  exactly twice. Consequently,
  $$
  \nu\left\{\MP\left(\lambda\right) \leq \alpha\pi \Biggl|
    \lambda-{\O\over{2s}}\in{2\pi\over\delta}\left[k-1,k\right]
  \right\} = \frac{2\alpha\pi}{2\pi}=\alpha.
  $$

  Similarly, when $\lambda$ varies between
  ${\O\over{2s}}+{2\pi n\over\delta}$ and
  ${\O\over{2s}}+{2\pi(n+\zeta)\over\delta}$,
  we have
  $$
  \nu\left\{\MP\left(\lambda\right) \leq \alpha\pi \Biggl|
    \lambda-{\O\over{2s}}\in{2\pi\over\delta}\left[n,n+\zeta\right]
  \right\} \leq \frac{2\alpha\pi}{2\pi\zeta}\leq{\alpha\over\zeta}.
  $$
  Overall,
  $$
  \nu\left\{\MP\left(\lambda\right) \leq \alpha\pi\right\} \leq
  \alpha{n\over{n+\zeta}}+{\alpha\over\zeta}\frac{\zeta}{n+\zeta}=\alpha{n+1\over{n+\zeta}}\le 2\alpha, $$
  proving \eqref{eq:measure-argument-single}.

  It is clear from the above that
  $\left\{\lambda: \MP\left(\lambda\right)\leq\alpha\pi\right\}$ is a
  union of intervals, each of length $2\alpha \pi$, repeating with the
  period of $\frac{2\pi}{\delta}$.  Consequently the set
  $\left\{\lambda\in \left[{\O \over 2s}, {\O\over s} \right] :
    \MP\left(\lambda\right)\leq\alpha\pi \right\}$ is a union of at
  most $\left\lceil\frac{\O}{2s}\frac{\delta}{2\pi}\right\rceil$
  intervals.  Since $\delta\leq\pi$ we have
  $\left\lceil\frac{\O}{2s}\frac{\delta}{2\pi}\right\rceil\leq\left\lceil{\O\over
      {4 s}}\right\rceil$, and so the set
  $\left\{\lambda\in \left[{\O \over 2s}, {\O\over s} \right] :
    \MP\left(\lambda\right) \leq \alpha\pi\right\}$ is a
  union of at most $ \left\lceil{\O\over {4 s}}\right\rceil $
  intervals.

  Now we put $\alpha_0={1-\xi \over s^2}$ and apply
  \eqref{eq:measure-argument-single} for every pair $\left(t_j,y\right)$ where
  $j=1,\dots,s$ and $y\in\xvec\setminus\xvec^{(j)}$. Denote
$$
J:=\bigcup_{t_j,y\in\xvec\setminus \xvec^{(j)}} \left\{\lambda \in \left[{\O \over 2s}, {\O\over s} \right] : \MP^{(t_j,y)}(\lambda)\leq\alpha_0\pi\right\},
$$
then by the union bound we obtain
  \begin{equation}\label{eq:union-bound}
      \nu\left(J\right) \leq \sum_{t_j,y} 2\alpha_0 = 2{s\choose 2}{1-\xi \over s^2} < 1-\xi.
    \end{equation}

    Fixing $I$ as the complement of the above set,
    $I=\left[{\O \over 2s}, {\O\over s} \right] \setminus J$, we have
    that $I$ is of total measure greater or equal to
    $\xi {\O\over{2s}}$, and for every $\lambda\in I$ the estimate
    \eqref{eq:blown-up2} holds. Clearly $J$ is a union of at most
    ${s^2\over 2} \left\lceil{\O\over {4 s}}\right\rceil$ intervals.
\end{proof}

Fix $\xi={1\over 2}$ and consider the set $I$ given by
\prettyref{lem:blowup-lemma}. Let us also fix a finite and positive integer
$N$, and consider the set of $2N+1$ equispaced points in
$\left[-\O,\O\right]$:
$$
P_N:=\left\{k{\O\over N}\right\}_{k=-N,\dots,N}.
$$

\begin{prop}\label{prop:finite-n-blowup}
  If $N > 2s^3 \left\lceil{\O\over {4 s}}\right\rceil$, then
  $P_N \cap I \neq \emptyset$.
\end{prop}

\begin{proof}
  By \prettyref{lem:blowup-lemma}, the set $I^c$ consists of
  $K \leq {s^2\over 2} \left\lceil{\O\over {4 s}}\right\rceil$ intervals, and by \eqref{eq:union-bound} the total
  length of $I^c$ is at most ${\O\over{4s}}$. Denote the lengths of
  those intervals by $d_1,\dots,d_K$. The distance between neighboring
  points in $P_N$ is ${\O\over N}$, and therefore each interval
  contains at most ${d_j N\over\O}+1$ points. Overall, the interval
  $I^c$ contains at most
  $$
  \sum_{j=1}^K \left( {d_j N \over \O}+1\right) \leq {\O\over {4s}}{N\over \O} + K
  $$
  points from $P_N$, and since the total number of points in
  $\left[{\O\over{2s}},{\O\over s}\right]$ is at least ${N\over{2s}}$,
  we have
  $$
  \left| P_N \cap I \right| \geq {N\over {2s}}-{N\over
    {4s}}-K\geq{N\over{4s}}-{s^2\over 2} \left\lceil{\O\over {4 s}}\right\rceil > 0.
  $$
\end{proof}

\subsection{Square Vandermonde matrices}
\label{sub:square-vand}
Let $\xiv=\left(\xi_1,\dots,\xi_s\right)$ be a vector of $s$ pairwise
distinct complex numbers.  Consider the square Vandermonde matrix
\begin{equation}
  \label{eq:vand-def}
  \VV(\xiv):=
  \begin{bmatrix}
    1 & 1 & \dots & 1 \\
    \xi_1 & \xi_2 & \dots & \xi_s \\
    \xi_1^2 & \xi_2^2 & \dots & \xi_s^2 \\
    \vdots & \vdots & \ddots & \vdots \\
    \xi_1^{s-1} & \xi_2^{s-1} & \dots & \xi_s^{s-1}
  \end{bmatrix}.
\end{equation}

\begin{thm}[Gautschi, \cite{gautschi_inverses_1962}]
  For a matrix $A=(a_{i,j})\in\CC^{m\times n}$, let $\|A\|_{\infty}$ denote
  the $\ell_\infty$ induced matrix norm
  $$
  \|A\|_\infty:=\max_{1\leq i\leq m}\sum_{1\leq j \leq n}|a_{i,j}|.
  $$

  Then we have
  \begin{equation}
    \label{eq:gautschi-estimate}
    \|\VV^{-1}\left(\xiv\right)\|_\infty \leq \max_{1\leq i \leq s}\prod_{j\neq i}\frac{1+|\xi_j|}{|\xi_j-\xi_i|}.
  \end{equation}
\end{thm}

\begin{prop}\label{prop:vand-sing-estimate}
  Suppose that $\xiv=\left(\xi_1,\dots,\xi_s\right)$ is a vector of
  pairwise distinct complex numbers with
  $|\xi_j|=1$, $j=1,\dots,s$, and let $r\in\RR$ be arbitrary. Let
  \begin{equation}
    \label{eq:vandermonde-arbitrary-start}
    \VV\left(\xiv,r\right):=
    \begin{bmatrix}
      \xi_1^r & \xi_2^r & \dots & \xi_s^r \\
      \xi_1^{r+1} & \xi_2^{r+1} & \dots & \xi_s^{r+1} \\
      \xi_1^{r+2} & \xi_2^{r+2} & \dots & \xi_s^{r+2} \\
      \vdots & \vdots & \ddots & \vdots \\
      \xi_1^{r+s-1} & \xi_2^{r+s-1} & \dots & \xi_s^{r+s-1}
    \end{bmatrix}.
  \end{equation}

  For $1\leq j < k \leq s$, denote by $\delta_{j,k}$ the angular distance between
  $\xi_j$ and $\xi_k$:
  $$
  \delta_{j,k}:=\left| \Arg \left({\xi_j \over \xi_k}\right) \right| = \bigl|\Arg(\xi_j) - \Arg(\xi_k) \mod (-\pi,\pi] \bigr|.
  $$

  Then
  \begin{equation}
    \label{eq:vand-sigmamin}
    \sigma_{\min}\left(\VV\left(\xiv,r\right)\right) \geq {\pi^{1-s}\over\sqrt{s}} \min_{1\leq j \leq s}\prod_{k\neq j}\delta_{j,k}.
  \end{equation}
\end{prop}
\begin{proof}
  Clearly, the matrix $\VV\left(\xiv,r\right)$ can be factorized as
  $$
  \VV\left(\xiv,r\right) = \VV\left(\xiv,0\right) \times
  \diag\left\{\xi_1^r,\dots,\xi_s^r\right\}.
  $$
  Since $\VV\left(\xiv,0\right)=\VV\left(\xiv\right)$ as in
  \eqref{eq:vand-def}, using \eqref{eq:gautschi-estimate} we
  immediately have
  \begin{equation}
    \label{eq:vand-prop-1}
    \|\VV^{-1}\left(\xiv,r\right)\|_{\infty} \leq 2^{s-1}\max_{1\leq j\leq s}\prod_{k\neq j}|\xi_j-\xi_k|^{-1}.
  \end{equation}
  For any $|\theta|\leq{\pi\over 2}$ we have
  $$
  {2\over\pi}\left|\theta\right| \leq \sin \left|\theta\right|
  \leq \left|\theta\right|,
  $$
  and since for any $\xi_j\neq\xi_k$
  $$
  \left|\xi_j-\xi_k\right|=\left|1-{\xi_j\over\xi_k}\right|=2\sin\left|{1\over
      2}\Arg{\xi_j\over\xi_k}\right|=2\sin\left|{\delta_{j,k}\over 2}\right|,
  $$
  we therefore obtain
  \begin{equation}\label{eq:angular.to.euclidean}
  {2\over\pi}\delta_{j,k} \leq \left|\xi_j-\xi_k\right|\leq
  \delta_{j,k}.
  \end{equation}

  Plugging \eqref{eq:angular.to.euclidean} into \eqref{eq:vand-prop-1} we have
  $$
  \sigma_{\max}\left(\VV^{-1}\left(\xiv,r\right)\right)\leq\sqrt{s}\|\VV^{-1}\left(\xiv,r\right)\|_\infty
  \leq \sqrt{s}\pi^{s-1}\max_{1\leq j\leq s}\prod_{k\neq j}
  \delta_{j,k}^{-1},
  $$
  which is precisely \eqref{eq:vand-sigmamin}.
\end{proof}

\subsection{Proof of \prettyref{thm:main-theorem}}
\label{sub:theproof}
We shall bound $\sigma_{\min}\left(\VV_N\left(\xvec,\O\right)\right)$
defined as in \eqref{eq:discretized-vandermonde} for sufficiently
large $N$. For any subset $R\subset \left\{-N,\ldots,N\right\}$ let $\VV_{N,R}$,
be the submatrix of $\VV_N$ containing only the rows in $R$. By the
Rayleigh characterization of singular values, it is immediately
obvious that if $\left\{-N,\ldots,N\right\}=R_1 \cup \dots \cup R_P$ is any
partition of the rows of $\VV_N$ then
\begin{equation}
  \label{eq:interlacing-singvals}
  \sigma^2_{\min}(\VV_N)\geq \sum_{n=1}^P \sigma^2_{\min}(\VV_{N,R_n}).
\end{equation}

Let $I$ be the set from \prettyref{lem:blowup-lemma} for
$\xi={1\over 2}$. By \prettyref{prop:finite-n-blowup} we have that for
all $N>2s^3\left\lceil\frac{\O}{4s}\right\rceil$,
$I$ will contain a rational multiple of $\O$ of the form
$\lambda_N={\O\over N}m$ for some $m\in\mathbb{N}$.  

Consider the ''new'' nodes 
\begin{equation}\label{eq:unodes-def}
  u_{j,N}:=t_j {\O\over N}m=\lambda_N t_j, \tab j=1,\dots, s.
\end{equation}

Since $\lambda_N\in I$, we conclude by \prettyref{lem:blowup-lemma} that
for every $j=1,\dots,s$
\begin{align}
  \label{eq:unodes-1}
  \|u_{j,N}-u_{k,N}\|_{\mathbb{T}}&\geq {1\over 2s}(\Delta\O), &&\forall t_k\in \xvec^{(j)}\setminus \{t_j\}; \\
  \label{eq:unodes-2}
  \|u_{j,N}-u_{k,N}\|_{\mathbb{T}} & \geq {\pi\over 2s^2}, &&\forall t_k\in\xvec\setminus\xvec^{(j)}.
\end{align}

Since $\lambda_N\leq{\O\over s}$ it follows that $ms \leq N$.
Now consider the particular interleaving partition of the rows
$\left\{-N,\ldots,N\right\}$ by blocks $R_{-m},\dots,R_{-1},\allowbreak R_0,R_{1},\ldots,R_m$ 
of $s$ rows each, separated by $m-1$
rows between them (some rows might be left out):
\begin{align*}
  R_0 &= \left\{0,m,\dots,(s-1)m\right\}, \\
  R_1 &= \left\{1,m+1,\dots,(s-1)m+1\right\}, \\
  R_{-1} &= \left\{-1,-m-1,\dots,-(s-1)m-1\right\}, \\
      & \dots \\
  R_{m-1} &= \left\{m-1,2m-1,\dots,sm-1\right\}, \\
  R_{-m+1} &= \left\{-m+1,-2m+1,\dots,-sm+1\right\}. \\
\end{align*}
For $n=-m+1,\ldots,m-1$, each $\VV_{N,R_n}$ is a square
Vandermonde-type matrix as in \eqref{eq:vandermonde-arbitrary-start},
$$
\VV_{N,R_n}={1\over\sqrt{2N}}\VV\left(\xiv,n\right),
$$
with node vector
$$
\xiv=\left\{e^{\imath u_{j,N}}\right\}_{j=1}^s,
$$
where $u_{j,N}$ are given by \eqref{eq:unodes-def}. We apply
\prettyref{prop:vand-sing-estimate} with the crude bound obtained from
\eqref{eq:unodes-1} and \eqref{eq:unodes-2} above:
$$
\min_{1\leq j \leq s}\prod_{k\neq j}\delta_{j,k} \geq {1\over 2^{s-1}s^{2s-2}} \left(\Delta\O\right)^{\ell-1}
$$
and obtain
$$
\sigma_{\min}\left(\VV_{N,R_n}\right) \geq {\Cl{aux1}(s)\over\sqrt{2N}} \left(\Delta \O\right)^{\ell-1},\qquad
\Cr{aux1}(s):=\frac{1}{(2\pi)^{s-1}s^{2s-2}\sqrt{s}}.
$$
Now we use \eqref{eq:interlacing-singvals} to aggregate the bounds on $\sigma_{\min}$ for 
each square matrix $\VV_{N,R_n}$ and obtain  
$$
\lambda_{\min}\left(\VV_N^H \VV_N \right) = \sigma^2_{\min}\left(\VV_N\right) \geq (2m-1)
\frac{\Cr{aux1}^2}{2N}\left(\Delta\O\right)^{2\left(\ell-1\right)}. 
$$
Since $m=\frac{\lambda_{N} N}{\O} \ge \frac{\O N}{2s \O} = \frac{N}{2s}$ 
and since by assumption $N>2s^3$, we have that ${2m-1 \over {2N}} \ge \frac{1}{4s}$ and so
$$
\sigma^2_{\min}\left(\VV_N\right) \geq 
\frac{\Cr{aux1}^2}{4s}\left(\Delta\O\right)^{2\left(\ell-1\right)}. 
$$

This proves \eqref{eq:main-bound-finite-n-g} and
\eqref{eq:main-bound-thm} with
\begin{equation}\label{eq:main-c-n-def}
  \Cr{main-c-n}(s):=\frac{1}{2(2\pi)^{s-1}s^{2s-1}}.
\end{equation}

\subsection{Proof of \prettyref{thm:optimality}}
\label{sub:optimality}

Let $\ell,s,\Delta,\O$ be fixed, with $\Delta\O<\eta$, where $\eta$
will be specified during the proof below, and
$\Delta<{\pi\over{2\left(\ell-1\right)}}$. We shall exhibit a
$\left(\Delta,\rho',s,\ell,\tau'\right)$-clustered configuration
$\xvec_{\min}$ with certain $\rho',\tau'$, such that
\begin{equation}
  \label{eq:opt-toprove}
  \lambda_{\min}\left(\GG\left(\xvec_{\min},\O\right)\right)  \leq \Cr{upper}\cdot
  \left(\Delta\O\right)^{2\left(\ell-1\right)},
\end{equation}
for some constant $\Cr{upper}=\Cr{upper}\left(\ell\right)$.

Define $\xvec_{\ell,\Delta}=\{t_1,\dots,t_{\ell}\}$ to be the vector
of $\ell$ equispaced nodes separated by $\Delta$,
i.e. $t_j=j\Delta,\;j=1,\dots,\ell$. Let
$\GG^{(\ell,\ell)}=\GG\left(\xvec_{\ell,\Delta},\O\right)$ be the
corresponding $\ell\times\ell$ prolate matrix.

\begin{prop}
  There exists an absolute constant $0<\eta_1\ll 1$ and
  $\Cl{slepian-const}=\Cr{slepian-const}\left(\ell\right)$ such that
  whenever $\O\Delta\leq\eta_1$, we have
  \begin{equation}\label{eq:opt-1}
    \lambda_{\min}\left(\GG^{(\ell,\ell)}\right) \leq \Cr{slepian-const} \cdot \left(\O\Delta\right)^{2\left(\ell-1\right)}.
  \end{equation}
\end{prop}

\begin{proof}
  By Slepian's results \cite{slepian_prolate_1978} elaborated in
  \prettyref{sec:known-bounds}, there exists a constant $\eta'\ll 1$
  for which \eqref{eq:slepian-explicit-bound} holds for all $s$, in
  particular for $s=\ell$, whenever $W\leq \eta'$, i.e. whenever
  $\O\Delta\leq\eta_1:=2\pi\eta'$.
\end{proof}

We define $\xvec_{\min}$ to be the extension of
$\xvec_{\ell,\Delta}$ such that the remaining $s-\ell$ nodes are maximally
equally spaced between $-{\pi\over 2}$ and $0$, not including the
endpoints. Under the assumptions on $s,\ell,\Delta$ specified in
\prettyref{thm:optimality}, it is easy to check that the nodes
$t_1,\dots,t_{\ell}$ are between $0$ and ${\pi\over 2}$, while the
remaining nodes are separated at least by
\begin{equation}\label{eq:rho-particular}
  \rho':={\pi\over{2\left(s-\ell+1\right)}}.
\end{equation}
Therefore, $\xvec_{\min}$ is a particular
$\left(\Delta,\rho',s,\ell,\tau'\right)$-clustered configuration
according to \prettyref{def:partial-cluster}, with $\rho'$ given by
\eqref{eq:rho-particular} and $\tau':=\ell-1$.

It is clear that $\GG^{\left(\ell,\ell\right)}$ is a principal
submatrix of $\GG\left(\xvec_{\min},\O\right)$, and therefore we can
apply the interlacing theorem for eigenvalues of partitioned Hermitian
matrices \cite[Theorem 4.3.28]{horn_matrix_2012}. Together with
\eqref{eq:opt-1}, this concludes the proof of \eqref{eq:opt-toprove}
and of \prettyref{thm:optimality} with $\Cr{upper}=\Cr{slepian-const}$
and $\eta=\eta_1$. \qed

\subsection{Proof of \prettyref{thm:minimax}}
\label{sub:proof-minimax}

By the definition of the matrix $\GG$ and \eqref{eq:omega-l2-norm}, we
immediately obtain the following fact.

\begin{prop}\label{prop:ft-norm-via-prolate}
  For $\mu=\sum_{j=1}^s a_j \delta_{t_j}\in\grdM$,
  denote $\xvec = \xvec_{\mu} := \supp \mu = \left(t_1,\dots,t_s\right)\in \RR^{s}$
  and $\cvec = \cvec_{\mu}:=\left(a_1,\dots,a_s\right)\in\CC^s$. Then we have
  $$
  \|\widehat{\mu}\|_{2,\O}^2 = \cvec^* \GG\left(\xvec,\O\right) \cvec.
  $$
\end{prop}

The next result shows that for any two measures with $s$ nodes and
clusters of size $\ell$, their difference has clusters of size at most
$2\ell$, provided that the grid size is small enough. Note that it may happen
that some nodes are in the support of both measures, in which case the
difference measure will have less than $2s$ nodes, and the largest
cluster may be of size strictly smaller than $2\ell$.

\begin{lem}\label{lem:merging.two}
  Fix $s,\ell,\rho,\tau$, and let there be given $K \ge 2$. Then there
  exists $\Delta_0$ such that for all $\Delta\leq \Delta_0$ the following
  holds: for any
  $\mu_1,\mu_2 \in \grdM\left(\Delta,\rho,s,\ell,\tau\right)$ we have
  $$
  \mu_1-\mu_2 \in \grdM\left(\Delta,\rho',s',\ell',\tau'\right)
  $$
  for some $\ell'\leq 2\ell,\;s'\leq 2s$, some $\rho',\tau'$
  satisfying $\rho' = K \tau' \Delta$ and $\tau' \geq 1$.
\end{lem}
\begin{proof}
  Let $\Delta\leq \Delta_0$ be given (with $\Delta_0$ to be determined
  below), and put $\rho_0:=\tau\Delta$.

  Consider the
  intervals $I_0,I_1,\dots,I_{2s+2}$ (see
  \prettyref{fig:merge-clusters}):
  \begin{align*}
    I_0 &= \left[0,\rho_0\right], \\
    I_1 &= \left[\rho_0,K\rho_0\right], \\
    I_2 &= \left[K\rho_0, K^2\rho_0\right], \\
    & \dots
  \end{align*}
  such that the length $a_j$ of each $I_j$ is
  \begin{align*}
    a_0 &= |I_0| = \rho_0;\\
    a_j &= |I_j| = (K-1) \cdot K^{j-1} \rho_0,\quad j \geq 1.
  \end{align*}
  One can verify that
  \begin{align*}
  a_n &= (K-1)\sum_{j=0}^{n-1} a_j, \quad n\ge 1;\\
  \sum_{j=0}^{n} a_j &= K^n \rho_0,\quad n \ge 0.
  \end{align*}

  The overall length of the $2s+3$ intervals is therefore
  $K^{2s+2}\rho_0$. From now on we assume that this quantity is at
  most ${\rho\over 2}$, which in particular means that
  $\Delta \leq \Delta_0:= {\rho\over{2\tau K^{2s+2}}}$.

  For each $t_j \in \supp \mu_1$, consider the following sets of distances:
  \begin{align*}
    {\cal R}_j &:= \{\|y-t_j\|_{\mathbb{T}}:\;y\in\supp \mu_1\} \cap
                 \left[0,{\rho\over 2}\right],\\
    {\cal S}_j &:=\{\|y-t_j\|_{\mathbb{T}}:\;y\in\supp \mu_2\} \cap
                 \left[0,{\rho\over 2}\right].
  \end{align*}
  It is obvious that
  $$
  \left(\bigcup_{j=1}^s {\cal R}_j \right) \cap
  \left(\bigcup_{j=2}^{2s+2} I_j \right) = \emptyset.
  $$
  
  Now, since $\mu_2 \in \grdM$, each one of the sets
  ${\cal S}_1,\dots,{\cal S}_s$ intersects at most two of the
  intervals $I_2,\dots,\allowbreak I_{2s+2}$.  Therefore, by the pigeonhole's
  principle (Dirichlet's principle), there exists an index
  $C\in\left\{2,\dots,2s+2\right\}$ such that
  $$
  I_C \bigcap \left(\bigcup_{j=1}^{s} {\cal S}_{j}\right) = \emptyset.
  $$
  Clearly we also have
  $I_C \bigcap \left(\bigcup_{j=1}^{s} {\cal R}_{j}\right) =
  \emptyset$. Put $ I_C = [a,b]$. The proof is finished by taking
  $\tau':={a\over\Delta} \geq {K\rho_0\over\Delta}=\tau\geq 1$ and
  $\rho':=b = Ka = K\tau'\Delta$.
\end{proof}

\begin{figure}
  \centering
  \includegraphics[width=\linewidth]{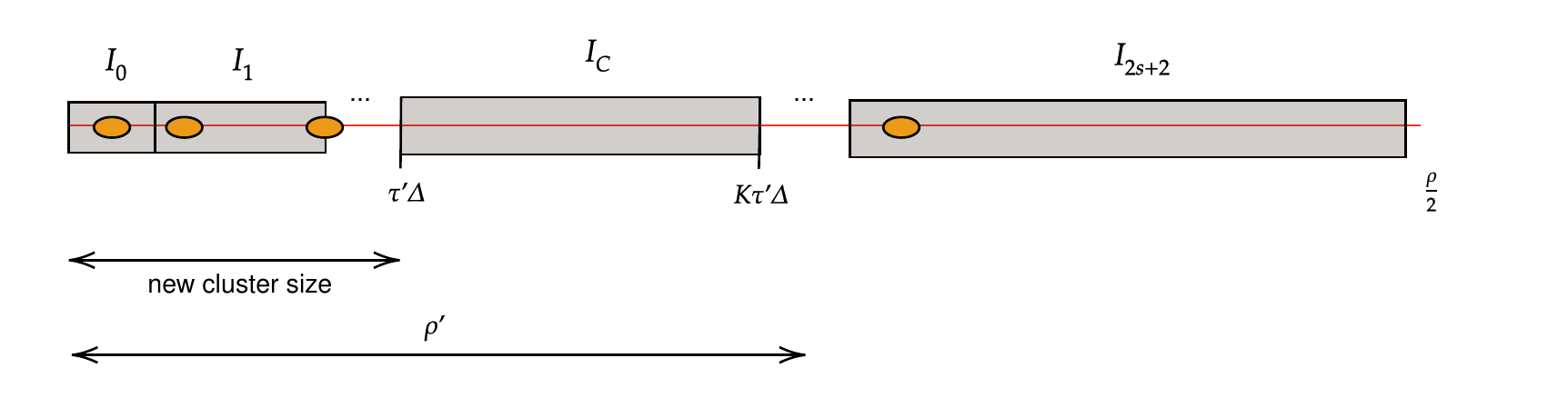}
  \caption{Merging $\mu_1$ and $\mu_2$ - see
    \prettyref{lem:merging.two}. The orange circles are the sets
    ${\cal S}_j$, $j=1,\dots,s$.}
  \label{fig:merge-clusters}
\end{figure}

\begin{proof}[Proof of upper bound] Let $s,\ell,\tau,\rho$ and
  $\varepsilon$ be fixed. Put $K:=8 s M$ and let $\Delta_0$ be as
  specified in \prettyref{lem:merging.two}. Let $\Delta\leq\Delta_0$,
  and put $\grdM = \grdM\left(\Delta,\rho,s,\ell,\tau\right)$ as in
  \prettyref{def:R-class}.

  Since the set $\grdM$ is finite, it is clearly possible to enumerate
  all its elements.  To prove the upper bound for the minimax error rate $\minmax$, consider
  the following estimator (clearly realizable, but computationally
  intractable) function
  $\widetilde{\mu}_0=\mu_{\grdM,\O,\varepsilon}:L_2\left([-\O,\O]\right)
  \to \grdM$:
$$
\widetilde{\mu}_0 \left(\Phi\right) := \left\{\textrm{the first }
\mu\in\grdM \quad \textit{ subject to } \|\Phi -
\widehat{\mu}\|_{2,\O} \leq \varepsilon\right\}.
$$

Given $\mu\in\grdM$ and $e \in L_2\left([-\O,\O]\right)$ with
$\|e\|_{2,\O} \leq \varepsilon$, let $\widetilde{\mu}_0 =
\widetilde{\mu}_0\left(\Phi_{\mu,e}\right)$ where $\Phi_{\mu,e}$ is
given by \eqref{eq:measurement-fun}. Then, since
$\|\Phi_{\mu,e}-\widehat{\mu}\|_{2,\O} = \|e\|_{2,\O} \leq
\varepsilon$ and also by the definition of $\widetilde{\mu}_0$, we
have
$$
\|\widehat{\widetilde{\mu}}_0 - \widehat{\mu}\|_{2,\O} \leq
2\varepsilon.
$$

Denote $\mu_2 := \widetilde{\mu}_0 - \mu$. By
\prettyref{lem:merging.two}, we get that
$\mu_2 \in \grdM\left(\Delta,\rho',s',\ell',\tau'\right)$ where
$s'\leq 2s$, $\ell'\leq 2\ell$, $\tau' \geq 1$ and
$\rho'=8 s M \tau'\Delta$. In particular, $\rho' > 4\tau'\Delta$, and
therefore by applying \prettyref{thm:main-theorem} we obtain that for all $\O$
satisfying
$$
{\pi s'\over{2sM\tau'\Delta}}={4\pi s'\over{\rho'}} \leq \O \leq {{\pi s'}\over{\tau'\Delta}},
$$
it holds that
$$
	\sqrt{\lambda_{\min}\left(\GG\left(\supp \mu_2,\O\right)\right)}\ge \Cr{main-c-n}(s') \left(\O \Delta\right)^{\ell'-1}. 
$$
In particular, for
${\pi\over {\Omega \Delta}}:=\alpha=M\tau' {2s\over s'} \geq M$, we 
have, using the above and \prettyref{prop:ft-norm-via-prolate}, that
\begin{align*}
  2\varepsilon \geq \|\widehat{\mu}_2\|_{2,\O} &\geq
\sqrt{\lambda_{\min}\left(\GG\left(\supp \mu_2,\O\right)\right)}
\|\mu_2\|_2 \\
  &\geq \Cr{main-c-n}(s') \left(\O \Delta\right)^{\ell'-1}\|\mu_2\|_2 \\
  &\geq \Cr{main-c-n}(2s) \left(\O \Delta\right)^{2\ell-1}\|\mu_2\|_2
\end{align*}
(here we also used that the constant $\Cr{main-c-n}(s)$ is
decreasing with $s$ and $\Delta\O < 1$).
This in turn proves that $\minmax \leq {2\over {\Cr{main-c-n}(2s) \pi^{2\ell-1}}}\srf^{2\ell-1}
\varepsilon$.
\end{proof}

\begin{proof}[Proof of the lower bound]
  Let $\eta$ be the constant from \prettyref{thm:optimality}, and put
  $\beta:={\pi\over\eta}$. Now suppose that $\alpha:=\srf>\beta$, that
  is, ${\Omega\Delta} < \eta$. Applying \prettyref{thm:optimality}
  with $2s,2\ell$ we obtain $\rho',\tau'$ and the minimal
  configuration $\xvec_{\alpha}$. Let $\cvec_{\alpha}\in\CC^{2s}$
  denote the corresponding minimal eigenvector of
  $\GG\left(\xvec_{\alpha},\O\right)$, with the normalization
  \begin{equation}\label{eq:min-eig-normalization}
\|\cvec_{\alpha}\|_2^2 = \frac{\varepsilon^2}{\Cr{upper}
\alpha^{2(2\ell-1)}}.
\end{equation}

Let $\mu$ be the measure defined by $\xvec_{\alpha},\cvec_{\alpha}$,
which therefore satisfies
$\mu \in \grdM\left(\Delta,\rho',2s,2\ell,\tau'\right)$.  By
\prettyref{prop:ft-norm-via-prolate}, \eqref{eq:min-eig-normalization}
and \prettyref{thm:optimality} we have
  $$
  \|\widehat{\mu}\|_{2,\O}^2 = \lambda_{\min}\left(\GG\left(\xvec_{\alpha},\O\right)\right) \|\cvec_{\alpha}\|_2^2 \leq \varepsilon^2.
  $$

  Clearly it is possible to write $\mu = \mu_1 - \mu_2$ where
  $\mu_1,\mu_2 \in \grdM\left(\Delta,\rho',s,\ell,\tau'\right)$. Let
  the measurement function $\Phi$ be such that
  $\Phi(\w):=\widehat{\mu}_2(\w),\;\w\in\left[-\O,\O\right]$, and let
  $\widetilde{\mu} = \widetilde{\mu}(\Phi) \in
  \grdM\left(\Delta,\rho',s,\ell,\tau'\right)$. Clearly $\Phi =
  \Phi_{\mu_1,-\widehat{\mu}}=\Phi_{\mu_2,0}$ (as per
  \eqref{eq:measurement-fun}), while also
\begin{align*}
\frac{\varepsilon}{\Cr{upper}^{1\over 2} \alpha^{2\ell-1}} & = \|\cvec_{\alpha}\|_2 = \|\mu\|_2 = \|\mu_1-\mu_2\|_2 \\
& \leq \|\mu_1-\widetilde{\mu}\|_2 + \|\mu_2 - \widetilde{\mu}\|_2 \\
& \leq 2\max\left(\|\mu_1-\widetilde{\mu}\|_2, \|\mu_2-\widetilde{\mu}\|_2\right),
\end{align*}
which shows \eqref{eq:new-minimax-lower} with $C_{\ell} = {1\over 2}
\Cr{upper}^{-{1\over 2}}$.
\end{proof}

\section{Numerical experiments}
\label{sec:Numerical-evidence}

In order to validate the bounds of \prettyref{thm:main-theorem} and
\prettyref{thm:optimality}, we computed
$\lambda_{\min}\left(\GG\right)$ for varying values of
$\Delta,\O,\ell,s$ and the actual clustering configurations. As
before, we put $\srf:={\pi\over{\Delta\O}}$. We checked two clustering
scenarios:
\begin{enumerate}
\item[\bf{C1}] A single equispaced cluster of size $\ell$ in
  $\left[\Delta,\ell\Delta\right]$, with the rest of the nodes equally
  spaced and maximally separated in $\left(-{\pi\over
      2},0\right)$. For example, in the case $s=8,\;\ell=4$ (as in
  \prettyref{subfig:c1-schem}) we have $t_j=j\Delta$ for
  $j=1,\dots,4$, and $t_j=-{\pi\over 2} + (j-4){\pi\over 10}$ for
  $j=5,\dots,8$.
\item[\bf{C2}] Split the $s$ nodes into two groups, and construct two
  single-clustered configurations as follows:
  \begin{enumerate}
  \item $s_1=\left\lfloor{s\over 2}\right\rfloor$ nodes, a single
    equispaced cluster of size $\ell_1=\ell$ in
    $\left[\Delta,\ell \Delta\right]$, and the rest of the
    $s_1-\ell_1$ nodes maximally separated and equally spaced in
    $\left(\ell\Delta,{\pi\over 2}\right)$;
  \item $s_2=s-s_1$ nodes, a single equispaced cluster of size
    $\ell_2=\ell$ in
    $\left[-{\pi\over 2}+\Delta,-{\pi\over 2}+\ell\Delta\right]$, and
    the rest of the $s_2-\ell_2$ nodes maximally separated and
    equally spaced in $\left(-{\pi\over 2}+\ell\Delta, 0\right)$.
  \end{enumerate}
  For example, in the case $s=5,\;\ell=2$ (as in
  \prettyref{subfig:c2-schem}) we have $t_1=\Delta,\;t_2 = 2\Delta$ and
  $t_3=-{\pi\over 2}+\Delta,\;t_4=-{\pi\over 2}+2\Delta,\;t_5 =
  -{\pi\over 4} + \Delta$.
\end{enumerate}

In each experiment we fixed $\ell,s$ and one of the scenarios above,
and run $n=1000$ random tests with $\Delta,\O$ randomly chosen within
appropriate ranges for each experiment. The results are presented
\prettyref{fig:decay-good}.

\begin{figure}[t]
  \centering \subfloat[$s=8,\ell=4$, 1 cluster (configuration
  {\bf{C1}}).]{\includegraphics[width=0.4\linewidth]{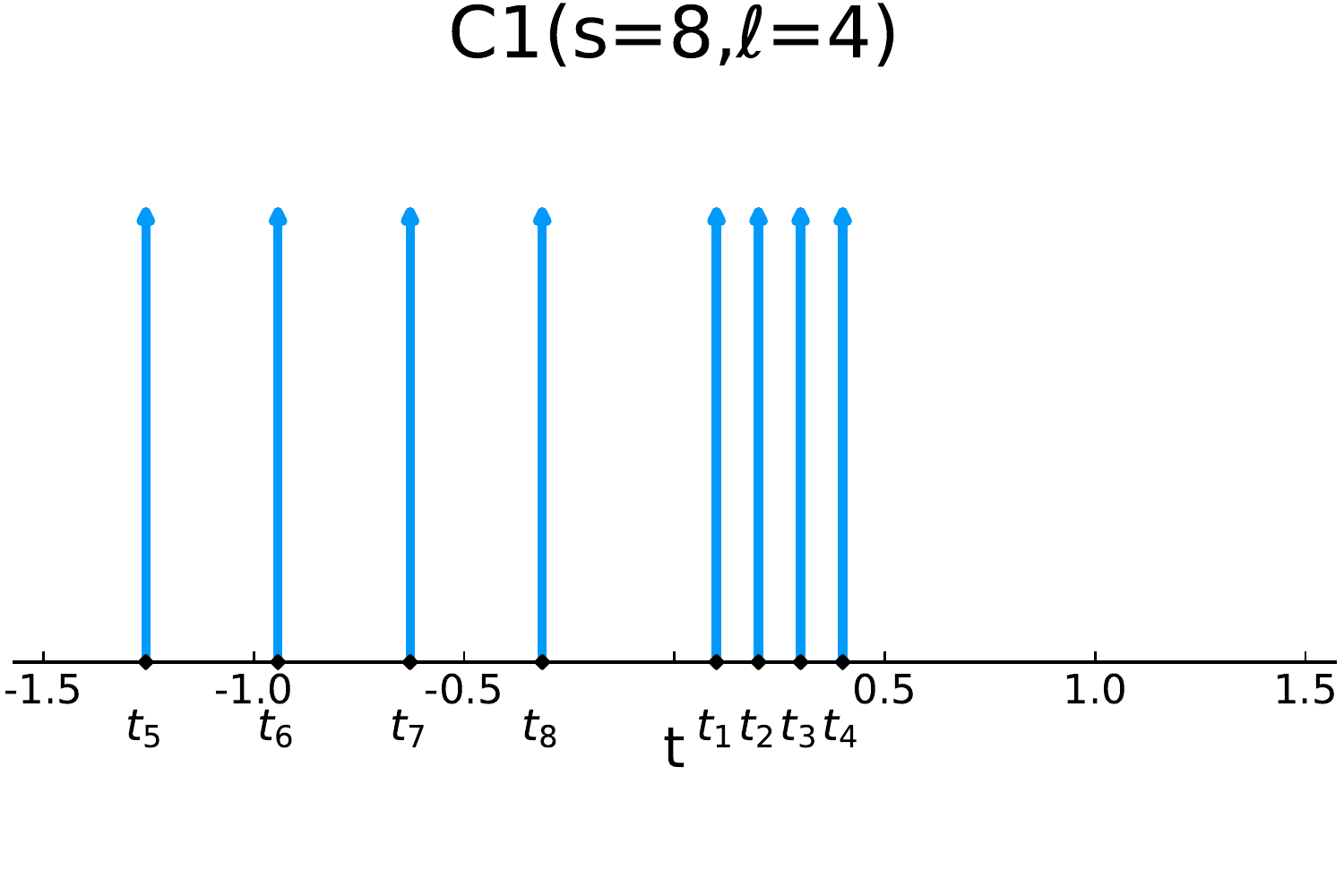}\label{subfig:c1-schem}}
  \subfloat[$s=5,\ell=2$, 2 clusters (configuration
  {\bf{C2}}).]{\includegraphics[width=0.4\linewidth]{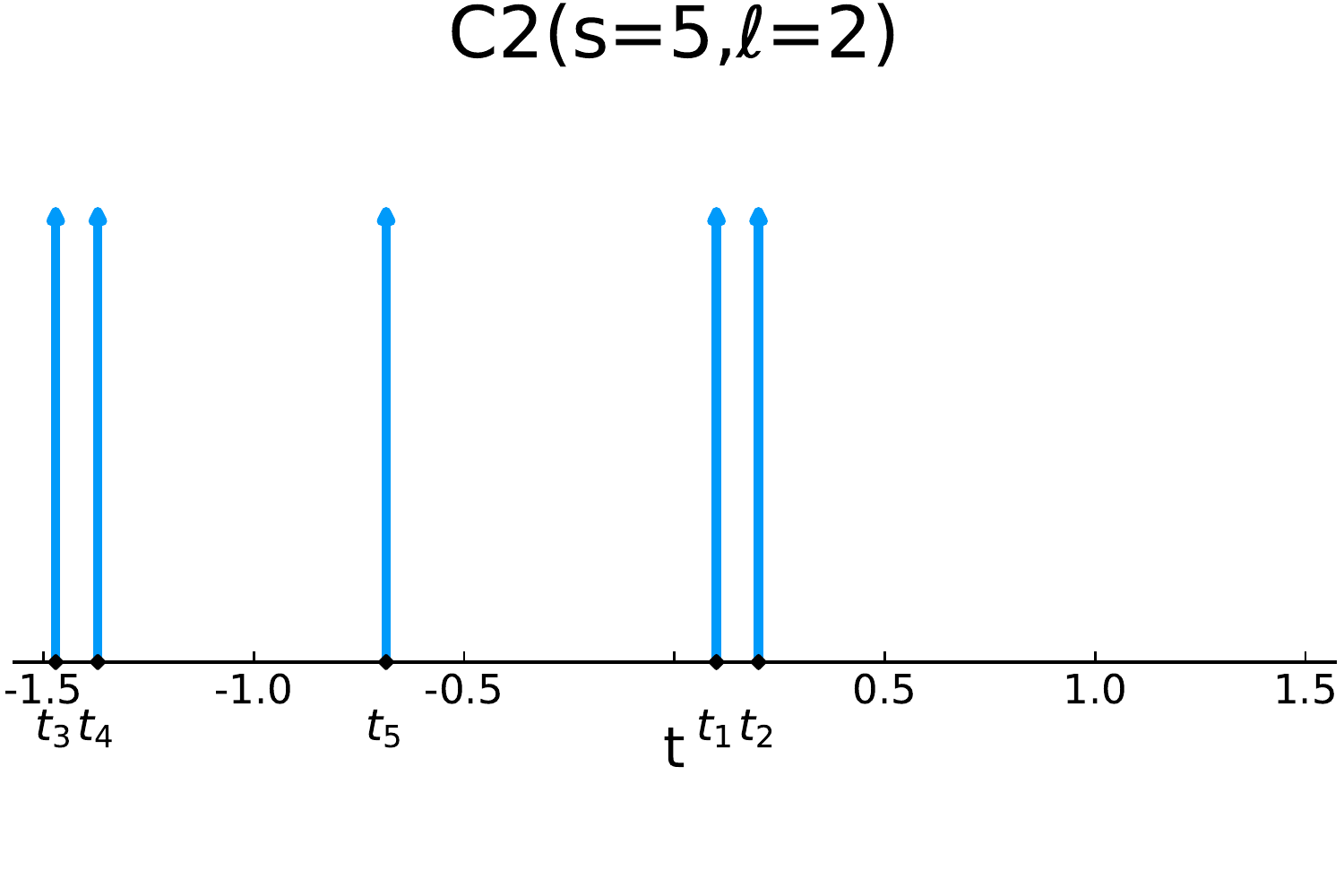}\label{subfig:c2-schem}}
  \caption{\small Examples for the configurations {\bf{C1}} and {\bf{C2}}.}
  \label{fig:configs-examples}
\end{figure}

\begin{figure}[t]
  \centering \subfloat[$s=8,\ell=4$, 1 cluster (configuration
  {\bf{C1}}).]{\includegraphics[width=0.5\linewidth]{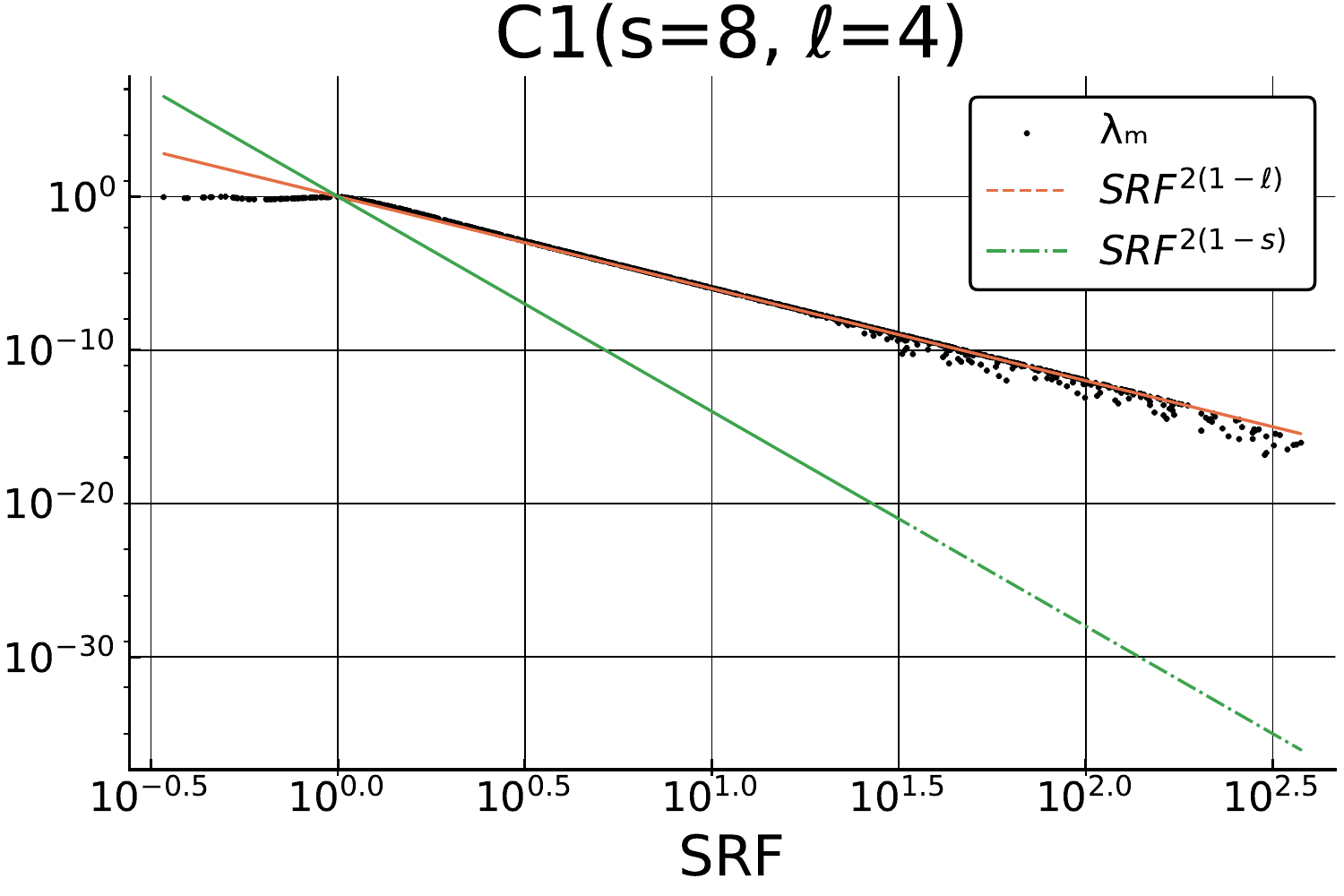}\label{subfig:c1}}
  \subfloat[$s=5,\ell=2$, 2 clusters (configuration
  {\bf{C2}}).]{\includegraphics[width=0.5\linewidth]{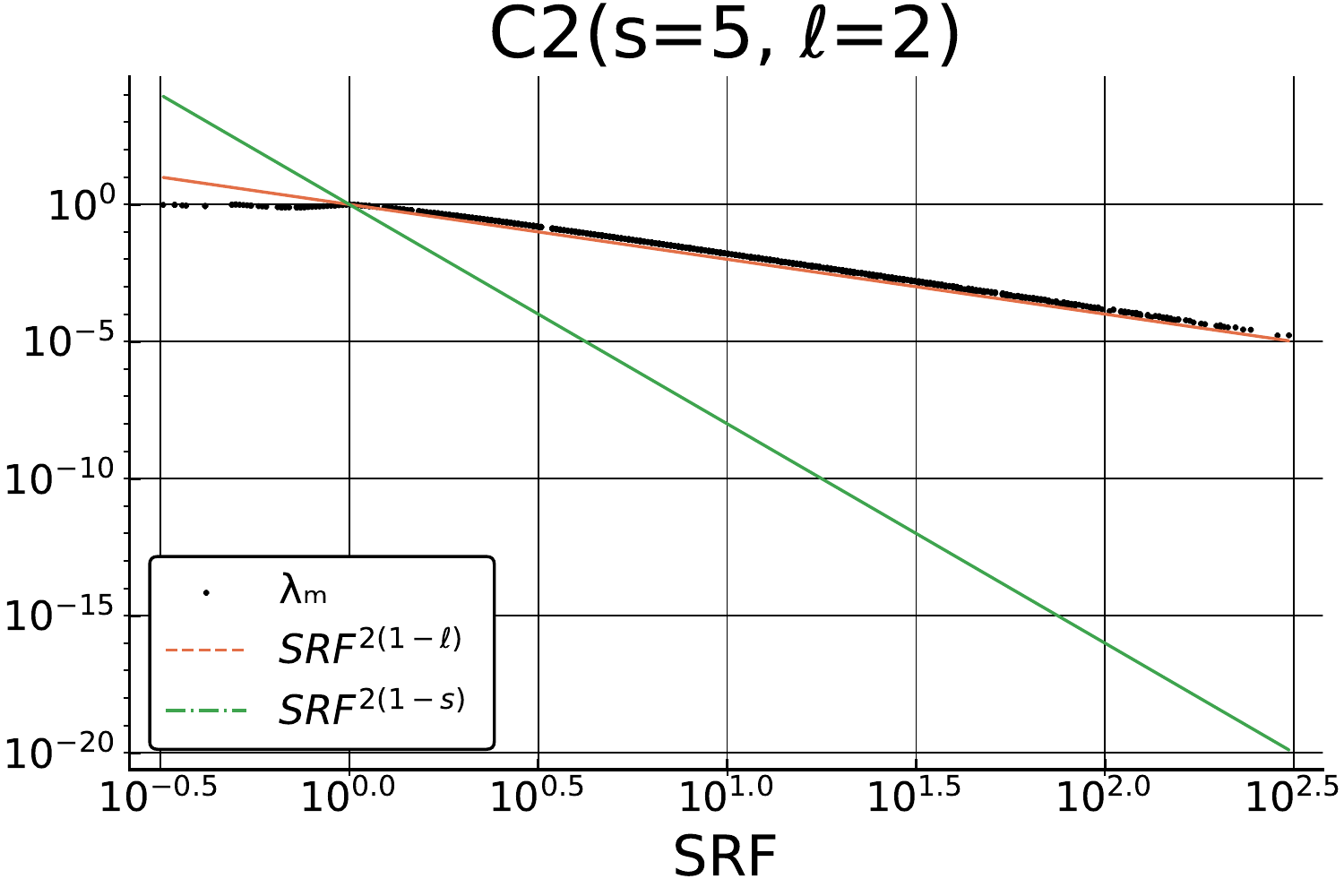}\label{subfig:c2}}
  \caption{\small Decay rate of $\lambda_{\min}$ as a function of
    $\srf$. Results of $n=1000$ random experiments with randomly
    chosen $\Delta,\O$ are plotted versus the theoretical bound
    $\srf^{2\left(1-\ell\right)}$. The curve
    $\srf^{2\left(1-s\right)}$ is shown for comparison. The bound
    stops to be accurate for $\srf< O(1)$.}
  \label{fig:decay-good}
\end{figure}

In another experiment (\prettyref{fig:breakdown}), we fixed
$\Delta,\ell,s$ and changed $\O$. As expected, when $\O$ became small
enough, the left inequality in \eqref{eq:Omega-req-th1} was violated,
and indeed we can see that in this case the asymptotic decay was
$\approx \srf^{2\left(1-s\right)}$. See
\prettyref{rem:different-confs} for further discussion.

\begin{figure}[thp]
  \begingroup \captionsetup[subfigure]{width=0.47\linewidth}
  \centering \subfloat[Configuration {\bf
    C1}.]{\includegraphics[width=0.5\columnwidth]{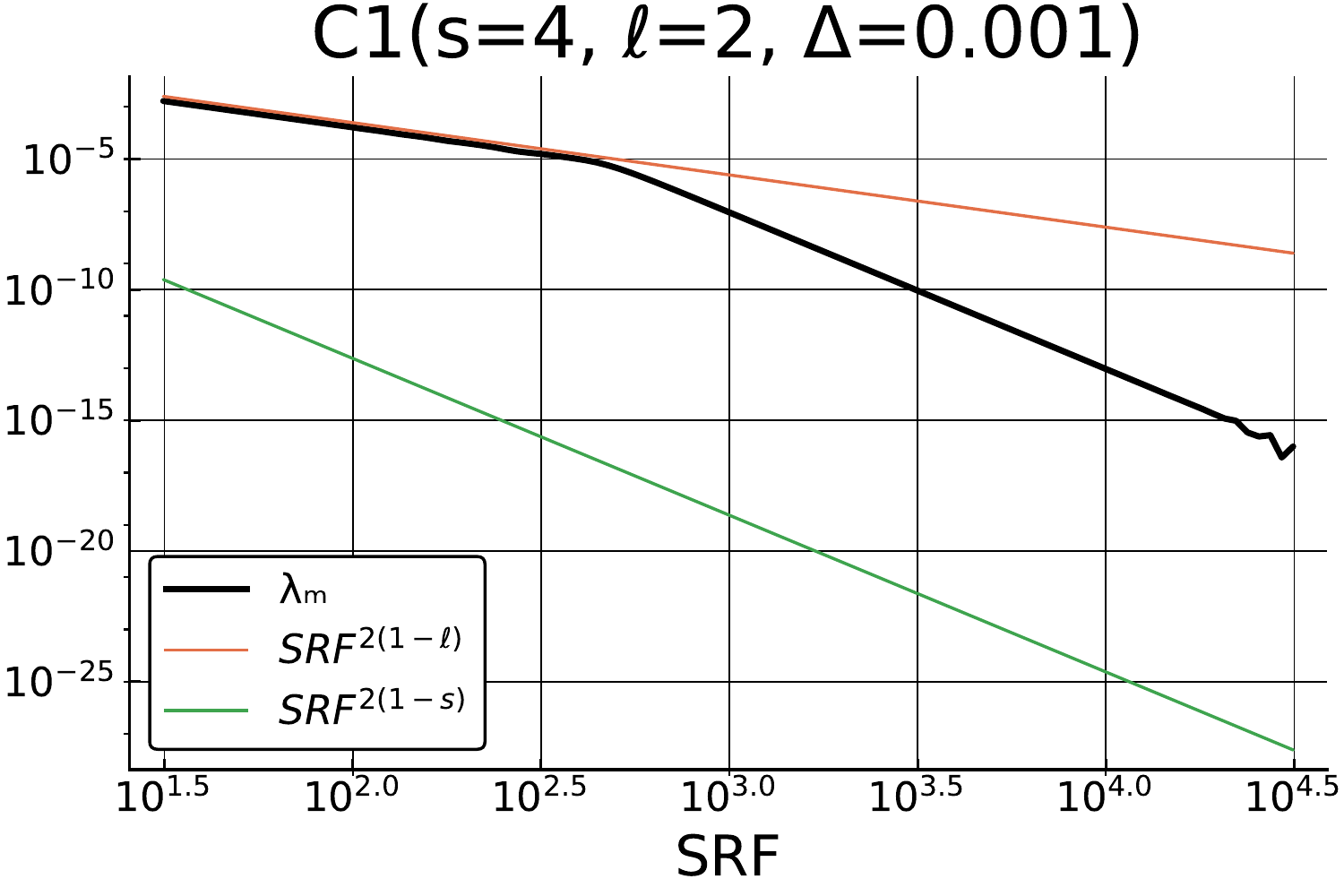}}
  \endgroup \begingroup \captionsetup[subfigure]{width=0.47\linewidth}
  \centering \subfloat[Configuration {\bf
    C2}.]{\includegraphics[width=0.5\columnwidth]{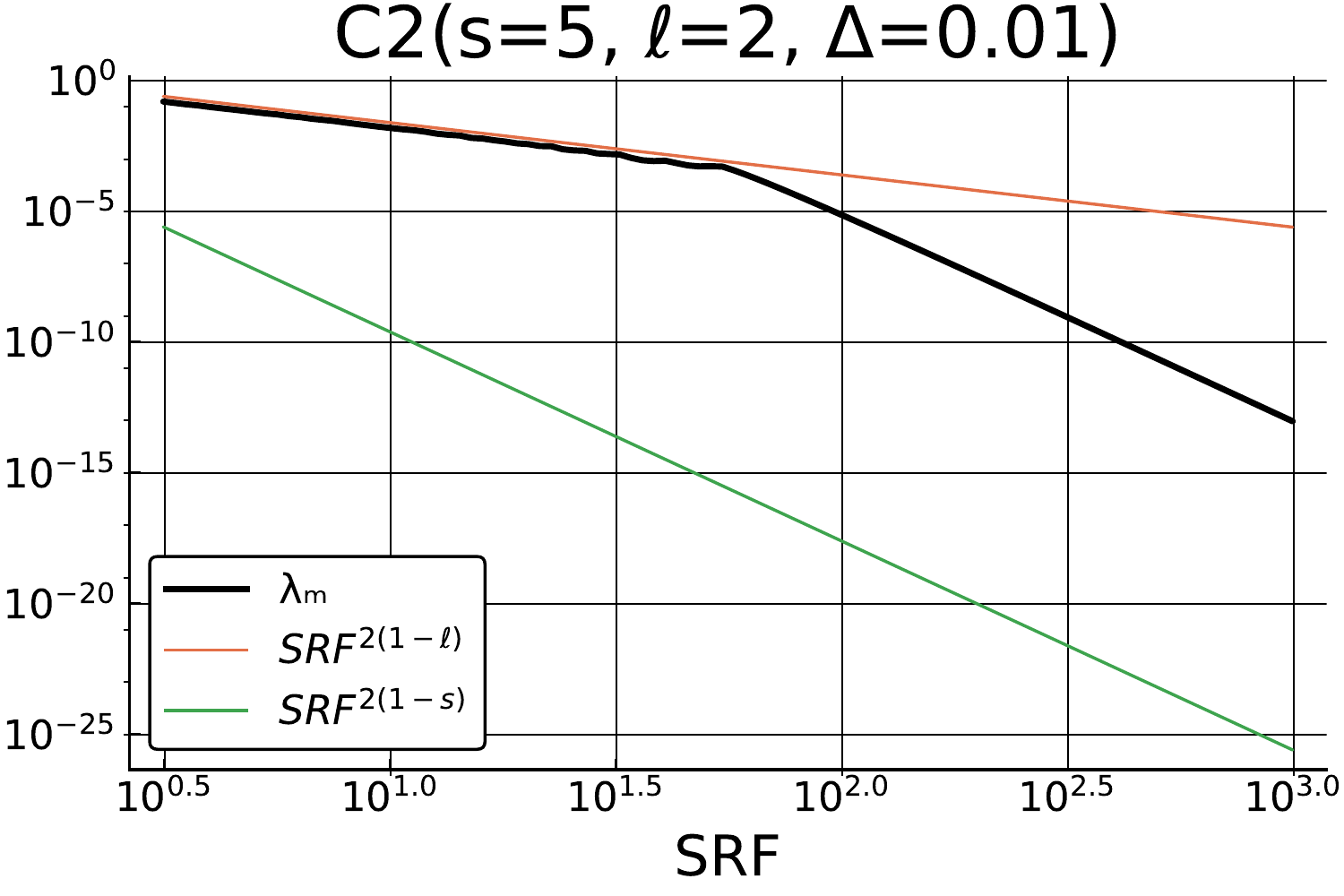}}
  \endgroup
  \caption{\small Breakdown of cluster structure. When $\O$ is small
    enough, the assumptions of \prettyref{thm:main-theorem} are
    violated for certain $\ell<s$. As a result, the decay rate of
    $\lambda_{\min}$ corresponds to the entire $\xvec$ being a single
    cluster of size $\ell=s$. $\Delta$ is kept fixed. See
    \prettyref{rem:different-confs}.}
  \label{fig:breakdown}
\end{figure}

\clearpage
\bibliographystyle{abbrv}
\bibliography{clustering-vandermonde}

\begin{thebibliography}{10}

\bibitem{akinshin_accuracy_2015}
A.~Akinshin, D.~Batenkov, and Y.~Yomdin.
\newblock Accuracy of spike-train {{Fourier}} reconstruction for colliding
  nodes.
\newblock In {\em 2015 {{International Conference}} on {{Sampling Theory}} and
  {{Applications}} ({{SampTA}})}, pages 617--621, May 2015.

\bibitem{akinshin2017geometry}
A.~Akinshin, G.~Goldman, and Y.~Yomdin.
\newblock Geometry of error amplification in solving {P}rony system with
  near-colliding nodes.
\newblock {\em arXiv preprint arXiv:1701.04058}, 2017.

\bibitem{aubel_vandermonde_2017}
C.~Aubel and H.~B\"olcskei.
\newblock Vandermonde matrices with nodes in the unit disk and the large sieve.
\newblock {\em Applied and Computational Harmonic Analysis}, Aug. 2017.

\bibitem{auton_investigation_1981}
J.~Auton.
\newblock Investigation of {{Procedures}} for {{Automatic Resonance
  Extraction}} from {{Noisy Transient Electromagnetics Data}}. {{Volume III}}.
  {{Translation}} of {{Prony}}'s {{Original Paper}} and {{Bibliography}} of
  {{Prony}}'s {{Method}}.
\newblock Technical report, {Effects Technology Inc., Santa Barbara, CA}, 1981.

\bibitem{batenkov_complete_2015}
D.~Batenkov.
\newblock Complete algebraic reconstruction of piecewise-smooth functions from
  {{Fourier}} data.
\newblock {\em Mathematics of Computation}, 84(295):2329--2350, 2015.

\bibitem{batenkov_accurate_2017}
D.~Batenkov.
\newblock Accurate solution of near-colliding {{Prony}} systems via decimation
  and homotopy continuation.
\newblock {\em Theoretical Computer Science}, 681:27--40, June 2017.

\bibitem{batenkov_stability_2016}
D.~Batenkov.
\newblock Stability and super-resolution of generalized spike recovery.
\newblock {\em Applied and Computational Harmonic Analysis}, 45(2):299--323,
  Sept. 2018.

\bibitem{superres_clusters18}
D.~Batenkov, G.~Goldman, and Y.~Yomdin.
\newblock Super-resolution of near-colliding point sources.
\newblock {\em arXiv:1904.09186 [math]}, Apr. 2019.

\bibitem{bazan_conditioning_2000}
F.~Baz\'an.
\newblock Conditioning of rectangular {{Vandermonde}} matrices with nodes in
  the unit disk.
\newblock {\em SIAM Journal on Matrix Analysis and Applications}, 21:679, 2000.

\bibitem{blu2008}
T.~Blu, P.-L. Dragotti, M.~Vetterli, P.~Marziliano, and L.~Coulot.
\newblock Sparse {{Sampling}} of {{Signal Innovations}}.
\newblock {\em IEEE Signal Processing Magazine}, 25(2):31--40, Mar. 2008.

\bibitem{candes_super-resolution_2013}
E.~J. Cand\`es and C.~{Fernandez-Granda}.
\newblock Super-{{Resolution}} from {{Noisy Data}}.
\newblock {\em Journal of Fourier Analysis and Applications}, 19(6):1229--1254,
  Dec. 2013.

\bibitem{candes_towards_2014}
E.~J. Cand\`es and C.~{Fernandez-Granda}.
\newblock Towards a {{Mathematical Theory}} of {{Super}}-resolution.
\newblock {\em Communications on Pure and Applied Mathematics}, 67(6):906--956,
  June 2014.

\bibitem{cuyt_sparse_2016}
A.~Cuyt, G.~Labahn, A.~Sidi, and W.-s. Lee.
\newblock Sparse modelling and multi-exponential analysis ({{Dagstuhl Seminar}}
  15251).
\newblock {\em Dagstuhl Reports}, 5(6):48--69, 2016.

\bibitem{de_castro_exact_2012}
Y.~{de Castro} and F.~Gamboa.
\newblock Exact reconstruction using {{Beurling}} minimal extrapolation.
\newblock {\em Journal of Mathematical Analysis and Applications},
  395(1):336--354, Nov. 2012.

\bibitem{demanet_recoverability_2014}
L.~Demanet and N.~Nguyen.
\newblock The recoverability limit for superresolution via sparsity.
\newblock 2014.

\bibitem{donoho_superresolution_1992}
D.~Donoho.
\newblock Superresolution via sparsity constraints.
\newblock {\em SIAM Journal on Mathematical Analysis}, 23(5):1309--1331, 1992.

\bibitem{duval_exact_2014}
V.~Duval and G.~Peyr\'e.
\newblock Exact {{Support Recovery}} for {{Sparse Spikes Deconvolution}}.
\newblock {\em Foundations of Computational Mathematics}, 15(5):1315--1355,
  Oct. 2014.

\bibitem{fannjiang_compressive_2016}
A.~Fannjiang.
\newblock Compressive {{Spectral Estimation}} with {{Single}}-{{Snapshot
  ESPRIT}}: {{Stability}} and {{Resolution}}.
\newblock {\em arXiv:1607.01827 [cs, math]}, July 2016.

\bibitem{ferreira_superresolution_1999}
P.~Ferreira.
\newblock Superresolution, the {{Recovery}} of {{Missing Samples}}, and
  {{Vandermonde Matrices}} on the {{Unit Circle}}.
\newblock 1999.

\bibitem{fomel_seismic_2013}
S.~Fomel.
\newblock Seismic data decomposition into spectral components using regularized
  nonstationary autoregression.
\newblock {\em GEOPHYSICS}, 78(6):O69--O76, Oct. 2013.

\bibitem{gautschi_inverses_1962}
W.~Gautschi.
\newblock On inverses of {{Vandermonde}} and confluent {{Vandermonde}}
  matrices.
\newblock {\em Numerische Mathematik}, 4(1):117--123, 1962.

\bibitem{horn_matrix_2012}
R.~A. Horn and C.~R. Johnson.
\newblock {\em Matrix Analysis}.
\newblock {Cambridge University Press}, Cambridge ; New York, 2nd ed edition,
  2012.

\bibitem{ingham_trigonometrical_1936}
A.~E. Ingham.
\newblock Some trigonometrical inequalities with applications to the theory of
  series.
\newblock {\em Mathematische Zeitschrift}, 41(1):367--379, Dec. 1936.

\bibitem{kunis2018}
S.~Kunis and D.~Nagel.
\newblock On the condition number of {{Vandermonde}} matrices with pairs of
  nearly-colliding nodes.
\newblock {\em arXiv:1812.08645 [math]}, Dec. 2018.

\bibitem{li_parametric_2000}
L.~Li and T.~P. Speed.
\newblock Parametric deconvolution of positive spike trains.
\newblock {\em Annals of Statistics}, pages 1279--1301, 2000.

\bibitem{li_stable_2017}
W.~Li and W.~Liao.
\newblock Stable super-resolution limit and smallest singular value of
  restricted {{Fourier}} matrices.
\newblock {\em arXiv:1709.03146 [cs, math]}, Sept. 2017.

\bibitem{li2019}
W.~Li, W.~Liao, and A.~Fannjiang.
\newblock Super-resolution limit of the {{ESPRIT}} algorithm.
\newblock {\em arXiv:1905.03782 [cs, math]}, May 2019.

\bibitem{li_phase_2015}
Y.~E. Li and L.~Demanet.
\newblock Phase and amplitude tracking for seismic event separation.
\newblock {\em Geophysics}, 80(6):WD59--WD72, 2015.

\bibitem{liao_music_2016}
W.~Liao and A.~Fannjiang.
\newblock {{MUSIC}} for single-snapshot spectral estimation: {{Stability}} and
  super-resolution.
\newblock {\em Applied and Computational Harmonic Analysis}, 40(1):33--67, Jan.
  2016.

\bibitem{moitra_super-resolution_2015}
A.~Moitra.
\newblock Super-resolution, {{Extremal Functions}} and the {{Condition Number}}
  of {{Vandermonde Matrices}}.
\newblock In {\em Proceedings of the {{Forty}}-{{Seventh Annual ACM}} on
  {{Symposium}} on {{Theory}} of {{Computing}}}, STOC '15, pages 821--830, New
  York, NY, USA, 2015. {ACM}.

\bibitem{montgomery_ten_1994}
H.~L. Montgomery.
\newblock {\em Ten {{Lectures}} on the {{Interface Between Analytic Number
  Theory}} and {{Harmonic Analysis}}}.
\newblock {American Mathematical Soc.}, 1994.

\bibitem{montgomery_hilberts_1974}
H.~L. Montgomery and R.~C. Vaughan.
\newblock Hilbert's {{Inequality}}.
\newblock {\em Journal of the London Mathematical Society}, s2-8(1):73--82, May
  1974.

\bibitem{morgenshtern2016}
V.~I. Morgenshtern and E.~J. Cand\`es.
\newblock Super-{{Resolution}} of {{Positive Sources}}: {{The Discrete Setup}}.
\newblock {\em SIAM Journal on Imaging Sciences}, 9(1):412--444, Jan. 2016.

\bibitem{negreanu_discrete_2006}
M.~Negreanu and E.~Zuazua.
\newblock Discrete {{Ingham Inequalities}} and {{Applications}}.
\newblock {\em SIAM Journal on Numerical Analysis}, 44(1):412--448, Jan. 2006.

\bibitem{pan_towards_2016}
H.~Pan, T.~Blu, and M.~Vetterli.
\newblock Towards {{Generalized FRI Sampling}} with an {{Application}} to
  {{Source Resolution}} in {{Radioastronomy}}.
\newblock {\em IEEE Transactions on Signal Processing}, 2016.

\bibitem{pereyra2010}
V.~Pereyra and G.~Scherer.
\newblock {\em Exponential {{Data Fitting}} and {{Its Applications}}}.
\newblock {Bentham Science Publishers}, Jan. 2010.

\bibitem{peter_nonlinear_2011}
T.~Peter, D.~Potts, and M.~Tasche.
\newblock Nonlinear approximation by sums of exponentials and translates.
\newblock {\em SIAM Journal on Scientific Computing}, 33(4):1920, 2011.

\bibitem{bini_error_2017}
D.~Potts and M.~Tasche.
\newblock Error {{Estimates}} for the {{ESPRIT Algorithm}}.
\newblock In D.~A. Bini, T.~Ehrhardt, A.~Y. Karlovich, and I.~Spitkovsky,
  editors, {\em Large {{Truncated Toeplitz Matrices}}, {{Toeplitz Operators}},
  and {{Related Topics}}}, volume 259, pages 621--648. {Springer International
  Publishing}, Cham, 2017.

\bibitem{reynolds_generalized_2013}
M.~Reynolds, G.~Beylkin, and L.~Monz{\'o}n.
\newblock On generalized {{Gaussian}} quadratures for bandlimited exponentials.
\newblock {\em Applied and Computational Harmonic Analysis}, 34(3):352--365,
  May 2013.

\bibitem{slepian_prolate_1978}
D.~Slepian.
\newblock Prolate spheroidal wave functions, {F}ourier analysis, and
  uncertainty -- {{V}}: The discrete case.
\newblock {\em Bell System Technical Journal, The}, 57(5):1371--1430, May 1978.

\bibitem{slepian_comments_1983}
D.~Slepian.
\newblock Some {{Comments}} on {{Fourier Analysis}}, {{Uncertainty}} and
  {{Modeling}}.
\newblock {\em SIAM Review}, 25(3):379--393, 1983.

\bibitem{stoica_spectral_2005}
P.~Stoica and R.~Moses.
\newblock {\em Spectral Analysis of Signals}.
\newblock {Pearson/Prentice Hall}, 2005.

\bibitem{varah_prolate_1993}
J.~Varah.
\newblock The prolate matrix.
\newblock {\em Linear Algebra and its Applications}, 187:269--278, July 1993.

\end{thebibliography}

\end{document}